\documentclass[a4paper,11pt]{article}

\usepackage{graphicx}
\usepackage{amsmath}
\usepackage{amsfonts}
\usepackage{titlesec}
\usepackage{amsthm}
\usepackage{amssymb}
\usepackage{geometry}
\usepackage{algorithm2e}

\newcommand{\R}{\mathbb{R}}
\newcommand{\T}{\mathcal{T}}

\newcommand{\LOD}{\mbox{\tiny{LOD}}}

\newcommand{\oVf}{\mathring{W}_{h}}
\newcommand{\Ic}{I_H}

\newtheorem{theorem}{Theorem}[section]

\newtheorem{lemma}[theorem]{Lemma}

\newtheorem{problem}[theorem]{Problem}
\theoremstyle{definition}
\newtheorem{definition}[theorem]{Definition}
\newtheorem{remark}[theorem]{Remark}
\newtheorem{conclusion}[theorem]{Conclusion} %

\graphicspath{{Grafiken/}}

\begin{document}

\begin{center}
{\LARGE Localized orthogonal decomposition techniques for boundary value problems\renewcommand{\thefootnote}{\fnsymbol{footnote}}\setcounter{footnote}{0}
 \hspace{-3pt}\footnote{This work was supported by The G\"{o}ran Gustafsson Foundation and The Swedish Research Council.}}\\[2em]
\end{center}

\renewcommand{\thefootnote}{\fnsymbol{footnote}}
\renewcommand{\thefootnote}{\arabic{footnote}}

\begin{center}
{\large Patrick Henning\footnote[1]{ANMC, Section de Math\'{e}matiques, \'{E}cole polytechnique f\'{e}d\'{e}rale de Lausanne, 1015 Lausanne, Switzerland},
Axel M\r{a}lqvist\footnote[2]{Department of Mathematical Sciences, Chalmers University of Technology and University of Gothenburg, SE-412 96 G\"oteborg, Sweden}}\\[2em]
\end{center}

\begin{center}
{\large{\today}}
\end{center}

\begin{center}
\end{center}

\begin{abstract}
In this paper we propose a Local Orthogonal Decomposition method (LOD) for elliptic partial differential equations with inhomogeneous Dirichlet- and Neumann boundary conditions. For this purpose, we present new boundary correctors which preserve the common convergence rates of the LOD, even if the boundary condition has a rapidly oscillating fine scale structure. We prove a corresponding a-priori error estimate and present numerical experiments. We also demonstrate numerically that the method is reliable with respect to thin conductivity channels in the diffusion matrix. Accurate results are obtained without resolving these channels by the coarse grid and without using patches that contain the channels.
\end{abstract}



\section{Introduction}
In this work we consider linear elliptic problems with a high variable diffusion matrix and possibly high variable Dirichlet and Neumann boundary conditions. Such problems are typically referred to as multiscale problems and arise in various applications, such as simulations of groundwater flow. Due to the large size of the computational domains and the rapid variations in the diffusivity which must be resolved by the computational grid, tremendous computing effort is needed. Such problems cannot be handled by standard finite element or finite volume methods.

To overcome these difficulties, a large number of so called multiscale methods have been proposed in the last decades (see e.g. \cite{Abdulle:2005,Abdulle:E:Engquist:Vanden-Eijnden:2012,MR2801210,E:Engquist:2003,Efendiev:Galvis:Hou:2013,Gloria:2011,Henning:Ohlberger:2009,Hou:Wu:1997,Owhadi:Zhang:Berlyand:2013,MR1660141,Larson:Malqvist:2007,Malqvist:2011,Ohlberger:2005,Owhadi:Zhang:2011}). In this work, we focus on the Local Orthogonal Decomposition Method (LOD) that was originally introduced in \cite{MaPe12} and that was derived from the Variational Multiscale Method (VMM) framework (c.f. \cite{Hughes:1995,Malqvist:2011}). 

The essence of the LOD is to construct a low dimensional solution space (with a locally supported partition of unity basis) that exhibits very high $H^1$-approximation properties with respect to the exact solution that we are interested in. The construction of the space does not rely on regularity or any structural assumptions on the type or the speed of the variations in the data functions. Advantages are therefore that the method does not rely on classical homogenization settings, but that it is also justified if no scale separation is available. The approach is fully robust against the variations in the diffusion matrix $A$. Furthermore, as shown in the numerical experiments, the method even shows a good behavior for high contrast cases and conductivity channels. Such structures typically have to be resolved with the coarse grid, but it is not necessary for the LOD. Like for most other multiscale methods, another advantage is that the constructed space (i.e. the computed correctors) can be fully reused for different source terms and different boundary conditions (in the latter case, only the boundary correctors have to be recomputed). This particularly pays off in various nonlinear settings, where the constructed space has to be computed only once, but can be reused in every iteration step of the nonlinear solver (see e.g. \cite{Henning:Malqvist:Peterseim:2012,Henning:Malqvist:Peterseim:2013}).

The fundamental idea of the LOD is to start from two computational grids: a coarse grid and a fine grid that resolves all fine scale features from the data functions. Accordingly, there are two corresponding finite element spaces, a coarse space $V_H$ and a very high dimensional space $V_h$. Introducing a quasi-interpolation operator $I_H : V_h \rightarrow V_H$, it is possible to define an (again high dimensional) remainder space $W_h$ that is just the kernel of the operator $I_H$. The orthogonal complement of $W_h$ in $V_h$ with respect to the energy scalar product is a low dimensional space with very high approximation properties (we refer to this space as the 'multiscale space' $V^{ms}$). With this strategy, it is possible to split the high dimensional finite element space $V_h$ into the orthogonal direct sum of a low dimensional multiscale space $V^{ms}$ and a high dimensional remainder space $W_h$. The final problem is solved in the low dimensional space $V^{ms}$ and is therefore cheap. However, the construction of the exact splitting of $V_h=V^{ms}\oplus W_h$ is computationally expensive and, therefore, must be somehow localized to make the method useful. In fact, localization is possible since the canonical basis functions of the multiscale space $V^{ms}$ show an exponential decay to zero outside of the support of the coarse finite element basis functions of $V_H$. A first localization strategy was proposed and analyzed in \cite{MaPe12}. Here, localized multiscale basis functions are determined by computing orthogonal complements of coarse basis functions in localized patches. This strategy has been recently applied to semi-linear multiscale problems \cite{Henning:Malqvist:Peterseim:2012}, eigenvalue problems \cite{Malqvist:Peterseim:2012}, the computation of ground states of Bose-Einstein condensates \cite{Henning:Malqvist:Peterseim:2013} and it was also combined with a discontinues Galerkin approach in \cite{Elverson:Georgoulis:Malqvist:Peterseim:2012,Elverson:Georgoulis:Malqvist:2013}. However, the localization strategy also suffers from a pollution of the exponential decay by the factor $1/H$, where $H$ denotes the coarse mesh size. This pollution is numerically visible and leads to larger patches for the localization problems. In \cite{Henning:Peterseim:2012}, motivated by homogenization theory, a different localization strategy was proposed, which successfully avoids the pollution effect and practically leads to much smaller patch sizes, which can be confirmed in numerical experiments. However, the localization proposed in \cite{Henning:Peterseim:2012} was given in a very specific formulation which is only adequate for finite element spaces consisting of piecewise affine functions on triangular meshes. In this paper we will close this gap by proposing a strategy that does not suffer from the mentioned pollution and that 
is applicable
to arbitrary coarse spaces $V_H$.

So far, inhomogeneous and mixed boundary conditions have been ignored in the construction and analysis of the LOD. High aspect ratios and channels have also not been studied in a systematic way.
In this work we extend and investigate the LOD further by:
\begin{enumerate}
\item introducing new boundary correctors that allow for an efficient treatment of inhomogeneous possibly oscillating Dirichlet and Neumann boundary conditions and
\item investigating the question of how the method reacts to high conductivity channels.
\end{enumerate}
These aspects are crucial for many multiscale applications, such as porous media flow where the porous medium might be crossed by cracks. Typically, these kind of structures have to be resolved with the coarse mesh in order to get accurate approximations. We will see that this is not necessary for the LOD. The approach that we propose will particularly generalize the localization strategy of \cite{Henning:Peterseim:2012}.
The new approach will no longer be restricted to triangular meshes and it will be also clear how the method generalizes to finite element spaces consisting of piecewise polynomials of a higher degree.

The general setting of this work is introduced in Section \ref{section-notation}, the LOD for boundary value problems is proposed in Section \ref{section-lod}, and detailed numerical experiments are given in Section \ref{section-numerical-experiments}.

\section{General setting and notation}
\label{section-notation}

In this paper, we consider a linear elliptic diffusion problem with mixed Dirichlet and Neumann boundary conditions, i.e. find $u$ with:
\begin{equation*}\label{eq:model}
  \begin{aligned}
    -\nabla \cdot A\nabla u &= f\quad \hspace{6pt}\text{in }\Omega, \\
    u & = g \quad \hspace{7pt}\text{on }\Gamma_D, \\
    A \nabla u \cdot n &= q \quad \hspace{7pt}\text{on }\Gamma_N,
  \end{aligned}
\end{equation*}
where
\begin{itemize}
\item[(A1)] $\Omega\subset\mathbb{R}^{d}$, for $d=1,2,3$, denotes a bounded Lipschitz domain with a piecewise polygonal boundary that is divided into two pairwise disjoint Hausdorff measurable submanifolds $\Gamma_{D}$ and $\Gamma_N$ with $\Gamma_D \cup \Gamma_N = \partial \Omega$ and $\Gamma_D$ being a closed set of non-zero Hausdorff measure of dimension $d-1$. By $n$ we denote the outward-pointing normal on $\partial \Omega$.
\item[(A2)] $f \in L^2(\Omega)$ denotes a given source term, $g \in H^{\frac{1}{2}}(\Gamma_D)$ the Dirichlet boundary values, and $q \in L^2(\Gamma_N)$ the Neumann boundary values.
\item[(A3)] $A\in L^\infty(\Omega,\mathbb{R}^{d\times d}_{sym})$ is a symmetric matrix-valued coefficient with uniform spectral bounds $\beta\geq\alpha>0$, 
\begin{equation}\label{e:spectralbound}
\sigma(A(x))\subset [\alpha,\beta]\quad\text{for almost all }x\in \Omega.
\end{equation}
\end{itemize}
Let $T_D : H^1(\Omega) \rightarrow H^{\frac{1}{2}}(\Gamma_D)$ denote trace operator with respect to $\Gamma_D$ and define the space $H^1_{\Gamma_D}(\Omega):=\{ v \in H^1(\Omega)| \hspace{2pt} T_D(v)=0\}$. Then, by the Lax-Milgram Theorem, there exists a unique weak solution of problem (\ref{eq:model}) above, i.e. $u \in H^1(\Omega)$ with $T_D(u)=g$ and
\begin{align*}
\int_{\Omega} A \nabla u \cdot \nabla \phi = \int_{\Omega} f \phi + \int_{\Gamma_N} q \phi \quad \mbox{for all } \phi \in H^1_{\Gamma_D}(\Omega). 
\end{align*}

In order to discretize the above problem, we consider two different shape-regular conforming triangulations/quadrilations $\T_H$ and $\T_h$ of $\Omega$. For instance, for $d=2$, both $\T_H$ and $\T_h$ consist either of triangles or 
quadrilaterals
and for $d=3$, both $\T_H$ and $\T_h$ consist either of 
tetrahedrons or 
hexahedrons.
We assume that $\T_h$ is a, possibly non-uniform, refinement of $\T_H$. By $H$ we denote the maximum diameter of an element of $\T_H$ and by $h\le H/2$ the maximum diameter of an element of $\T_h$. Together with $h\le H/2$, we also assume that $\T_H$ was at least one time globally (uniformly) refined to generate $\T_h$ (otherwise the usage of our approach does not make sense). The 'coarse scale' partition $\T_H$ is arbitrary whereas the 'fine scale' partition $\T_h$ is connected to the problem in the sense that we assume that the grid fully resolves the variations in the coefficients $A$ and $g$. For $\T=\T_H,\T_h$ we denote
\begin{align*}
P_1(\T) := \{v &\in C^0(\Omega) \;\vert \\
&\;\forall T\in\T: \hspace{2pt} v\vert_T \text{ is a polynomial of total (resp. partial) degree}\leq 1\}.
\end{align*}
We define $V_h:=P_1(\T_h)\cap H^1_{\Gamma_D}(\Omega)$ to be the 'high resolution' finite element space and $V_H:=P_1(\T_H)\cap H^1_{\Gamma_D}(\Omega) \subset V_h$ to be the coarse space. By $\mathcal{N}_H$ we denote the set of Lagrange points of the coarse grid $\T_H$ and by $\mathcal{N}_h$ we denote the set of the Lagrange points of the fine grid $\T_h$. For simplification, we assume that $\overline{\Gamma}_D \cap \overline{\Gamma}_N \subset \mathcal{N}_H$ (i.e. there is always a node on the interface between the Dirichlet and Neumann boundary segments).

In the following, the notation $a \lesssim b$ stands for $a\leq Cb$ with some constant $C$ that can depend on the space dimension $d$, $\Omega$, $\alpha$, $\beta$ and interior angles of the triangulations, but not on the mesh sizes $H$ and $h$. In particular it does not depend on the possibly rapid oscillations in $A$, $g$, $q$ and $f$.

\subsection{Reference problem}
We now define the fine scale reference problem. In the following, we do not compare the error between the exact solution and the LOD approximation (that we introduce in the next section), but always the error between LOD approximation and a fine scale reference solution. First, we need an approximation of the Dirichlet boundary condition: for each $z \in \mathcal{N}_h \cap \Gamma_D$ and $B_{\epsilon}(z)$ denoting a ball with radius $\epsilon$ around $z$, we define
$$g_z:=\lim_{\epsilon \rightarrow 0} \int_{\Gamma_D \cap B_{\epsilon}(z)} \hspace{-43pt}-\hspace{10pt} g\hspace{20pt}.$$
If $g$ is continuous we have $g_z=g(z)$. Now, let $g_H \in P_1(\mathcal{T}_H)$ be the function that is uniquely determined by the nodal values $g_H(z)=g_z$ for all $z \in \mathcal{N}_H \cap \Gamma_D$ and $g_H(z)=0$ for all $z \in \mathcal{N}_H \setminus \Gamma_D$. Using this, we define the (fine scale) Dirichlet extension $g_h \in P_1(\mathcal{T}_h)$ uniquely by the nodal values $g_h(z)=g_z$ for all $z \in \mathcal{N}_h \cap \Gamma_D$ and $g_h(z)=g_H(z)$ for all $z \in \mathcal{N}_h \setminus \Gamma_D$. With this, we avoid degeneracy of $g_h$ for $h$ tending to zero. The reference problem reads: find $v_h \in V_h$ with
\begin{align}
\label{reference-problem-eq} \int_{\Omega} A \nabla v_h \cdot \nabla \phi_h = \int_{\Omega} f \phi_h - \int_{\Omega} A \nabla g_h \cdot \nabla \phi_h +\int_{\Gamma_N} q \phi_h \quad \mbox{for all } \phi_h \in V_h. 
\end{align}
Define the final fine scale approximation by $u_h:=v_h+g_h$.

\section{Local Orthogonal Decomposition (LOD)}
\label{section-lod}

\subsection{Orthogonal Decomposition}
Let $\mathring{\mathcal{N}}_{H}:=\mathcal{N}_H \setminus \Gamma_D$ be the set of free coarse nodes.
For $z \in \mathcal{N}_H$ we let $\Phi_z \in V_H$ denote the corresponding nodal basis function with $\Phi_z(z)=1$ and $\Phi_z(y)=0$ for all $y\in \mathcal{N}_H \setminus \{ z\}$. We define a weighted Cl\'ement-type quasi-interpolation operator (c.f. \cite{MR1736895,MR1706735})
\begin{align}
\label{def-weighted-clement} I_H : H^1_{\Gamma_D}(\Omega) \rightarrow V_H,\quad v\mapsto I_H(v):= \sum_{z \in \mathring{\mathcal{N}}_{H}} v_z \Phi_z \quad \text{with }v_z := \frac{(v,\Phi_z)_{L^2(\Omega)}}{(1,\Phi_z)_{L^2(\Omega)}}.
\end{align}
Using that the operator $(I_H)_{|V_H}: V_H \rightarrow V_H$ is an isomorphism (see \cite{MaPe12}), we can define 
$W_h := \{ v_h \in V_h | \hspace{2pt} I_H(v_h) = 0 \}$ to
construct a splitting of the space $V_h$ into the direct sum
\begin{align}
\label{def-W_h}V_h = V_H  \oplus W_h,
\quad \mbox{where} \enspace \underset{\in V_h}{\underbrace{v_h}} = \underset{\in V_H}{\underbrace{({I_H}\vert_{V_H})^{-1}({I_H}(v_h))}} + \underset{\in W_h}{\underbrace{v_h - ({I_H}\vert_{V_H})^{-1}({I_H}(v_h))}}.
\end{align}
The subspace $W_h$ contains the fine scale features in $V_h$ that cannot be captured by the coarse space $V_H$. However, the fact that $W_h$ is the kernel of an interpolation operator suggests that the features of the (high dimensional) space $W_h$ could be neglected. Consequently we can look for a splitting $V_h = V_H^{\mbox{\tiny\rm new}} \oplus W_h$, where $V_H^{\mbox{\tiny\rm new}}$ has high $H^1$-approximation properties to the solution of the multiscale problem, but where
$V_H^{\mbox{\tiny\rm new}}$ is low dimensional because dim$(V_H^{\mbox{\tiny\rm new}})=$dim$(V_h)-$dim$(W_h)=$dim$(V_H)$. In order to explicitly construct such a splitting, we look for the orthogonal complement of $W_h$ in $V_h$ with respect to the scalar product $( A \nabla \cdot, \nabla \cdot )_{L^2(\Omega)}$. The corresponding orthogonal projection $P_{A,h}: V_h \rightarrow W_h$ is given by: for $v_h\in V_h$, $P_{A,h}(v_h)\in W_h$ solves
\begin{equation*}\label{e:finescaleproj}
( A\nabla P_{A,h}(v_h),\nabla w_h )_{L^2(\Omega)} =(  A\nabla v_h,\nabla w_h )_{L^2(\Omega)} \quad\text{for all }w_h \in W_h.
\end{equation*}
Observe that we have $(1-P_{A,h})(V_h)=(1-P_{A,h})(V_H)$ since $V_h = V_H \oplus W_h$ and $(1-P_{A,h})(W_h)=0$. We can therefore define
\begin{align}
\label{definition-of_V-H-c}V_H^c :=(1-P_{A,h})(V_H)
\end{align}
to obtain the desired splitting
\begin{align*}
V_h = \mbox{kern}(P_{A,h}) \oplus W_h = (1\hspace{-2pt}-\hspace{-2pt}P_{A,h})(V_h) \oplus W_h {=} (1\hspace{-2pt}-\hspace{-2pt}P_{A,h})(V_H) \oplus W_h = V_H^c \oplus W_h.
\end{align*}
Observe that this splitting can be equivalently characterized by a localized operator $\mathbf{Q}_h^{T} : V_h \rightarrow W_h$ with
$\mathbf{Q}_h^{T}(v_h) \in W_h$ solving
\begin{align}
\label{optimal-global-corrector}\int_{\Omega} A \nabla \mathbf{Q}_h^{T}(\phi_h)\cdot \nabla w_h = - \int_T A \nabla \phi_h \cdot \nabla w_h \qquad \mbox{for all } w_h \in W_h.
\end{align}
In this case we obtain that $P_{A,h} = - \sum_{T \in \mathcal{T}_H} \mathbf{Q}_h^{T}$. Since $\mathbf{Q}_h^{T}(\phi_h)$ decays rapidly to zero outside of $T$ (allowing to replace $\Omega$ by some small environment of $T$), the above reformulation of $P_{A,h} = - \sum_{T \in \mathcal{T}_H} \mathbf{Q}_h^{T}$ will be the basis for constructing a suitable localized version of the splitting $V_h = V_H^c \oplus W_h$. This will be done in the next subsection.

\subsection{Localization and formulation of the method}

In order to localize the 'detail space' $W_h$, we use admissible patches. We call this restriction to patches {\it localization}.

\begin{definition}[Admissible patch]\label{def-admissible-patches}
For $T \in \T_H$, we call $U(T)$ an {\it admissible patch} of
$T$, if it is non-empty, open, and connected, if $T \subset U(T) \subset \Omega$ and if
it is the union of the closure of elements of $\T_h$, i.e.
\begin{align*}
    U(T) = \operatorname{int}\bigcup_{{\tau} \in \T_h^{\ast}} \overline{\tau}, \quad \mbox{where} \enspace
    \T_h^{\ast} \subset \T_h.
\end{align*}
\end{definition}
By $\mathcal{U}$ we denote a given set of admissible localization patches, i.e.
\begin{align*}
 \mathcal{U} := \{ U(T)\;|\; \hspace{2pt} T \in \T_H \enspace \mbox{and } U(T) \enspace \mbox{is an admissible patch}\},
\end{align*}
where $\mathcal{U}$ contains one and only one patch $U(T)$ for each $T \in \T_H$. Throughout the paper, we refer to the set $U(T) \setminus T$ as an {\it extension layer}. Now, for any given admissible patch $U(T) \subset \Omega$ we define the restriction of $W_h$ to $U(T)$ by 
$$\mathring{W}_h(U(T)):=\{ v_h \in W_h| \hspace{2pt} v_h=0 \enspace \mbox{in } \Omega \setminus U(T) \}.$$ With this, we are prepared to define the local orthogonal decomposition method:

\begin{definition}[LOD approximation for boundary value problems]$\\$
\label{definition-loc-lod-approx}For a given set $\mathcal{U}$ of admissible patches, we define the local correction operator $Q_h^T: V_h \rightarrow \mathring{W}(U(T))$ by: for a given $\phi_h \in V_h$ and $T\in \T_H$ find $Q_h^{T}(\phi_h) \in \mathring{W}_h(U(T))$ such that
\begin{align}
\label{local-corrector-problem}\int_{U(T)} A \nabla Q_h^{T}(\phi_h)\cdot \nabla w_h = - \int_T A \nabla \phi_h \cdot \nabla w_h \qquad \mbox{for all } w_h \in \mathring{W}_h(U(T)).
\end{align}
The Neumann boundary correctors are given by: for all $T\in \T_H$ with $T \cap \Gamma_N\neq \emptyset$ find $B_{h}^{T} \in \mathring{W}_h(U(T))$ such that
\begin{align}
\label{local-boundary-corrector-problem}\int_{U(T)} A \nabla B_{h}^{T} \cdot \nabla w_h = - \int_{T\cap \Gamma_N} q w_h \qquad \mbox{for all } w_h \in \mathring{W}_h(U(T)).
\end{align}
The global correctors are given by
\begin{align*}
Q_h(\phi_h):=\sum_{T\in \T_H} Q_h^{T}(\phi_h) \quad \mbox{and} \quad B_{h}:=\underset{T\cap \Gamma_N\neq \emptyset}{\sum_{T\in \T_H}} B_{h}^{T}.
\end{align*}
Defining $R_h:=\mbox{Id} + Q_h$, the LOD approximation is given by $u_{\LOD}:=R_h(v_H + g_h)-B_{h}$, where $v_H \in V_H$ solves:
\begin{eqnarray}
\label{lod-problem-eq} \lefteqn{\int_{\Omega} A \nabla R_h(v_H) \cdot \nabla R_h(\Phi_H)}\\
\nonumber&=& \int_{\Omega} f  R_h(\Phi_H) - \int_{\Omega} A \nabla( R_h(g_h)-B_{h})\cdot \nabla R_h(\Phi_H) +\int_{\Gamma_N} q R_h(\Phi_H)
\end{eqnarray}
for all $\Phi_H \in V_H$.
\end{definition}
That problem (\ref{lod-problem-eq}) is well-posed follows by the Lax-Milgram theorem in the Hilbert space $X=\{ R_h(\Phi_H)| \Phi_H \in V_H \}$ and the fact that $\Phi_H= I_H( R_h(\Phi_H) )$ for all $\Phi_H \in V_H$.
\begin{remark}[Interpretation of the method for $U(T)=\Omega$]
Recall the definition of $V_H^c$ (see (\ref{definition-of_V-H-c})) and assume that $g=0$, $q=0$ and $U(T)=\Omega$ for all $T \in \mathcal{T}_H$. Then, $u_{\LOD} \in V_H^c$ is the unique solution of
\begin{align*}
\int_{\Omega} A \nabla u_{\LOD} \cdot \nabla \Phi = \int_{\Omega} f  \Phi \qquad  \mbox{for all } \Phi \in V_H^c.
\end{align*}
Furthermore, we have $u_h-u_{\LOD} \in \mbox{\rm{kern}}(I_H)=W_h$ and the explicit relation
$$u_{\LOD}= \left((1 - P_{A,h}) \circ ({I_H}\vert_{V_H})^{-1} \circ {I_H}\right) (u_h).$$
\end{remark}

Practically, using the fact that the basis functions of $V_H$ have a partition of unity property, we need to solve the local corrector problem \eqref{local-corrector-problem} only $d \cdot |\T_H|$ times in the case of a triangulation and $(d+1) \cdot |\T_H|$ times in the case of a quadrilation. Additionally, we need to determine the corrector $Q_h(g_h)$ which involves solving a local problem for each $T\in \T_H$ with $\overline{T}\cap \Gamma_D\neq \emptyset$.

Note that even though the method was defined for finite elements spaces of partial degree less than or equal to $1$, it directly generalizes to arbitrary polynomial degrees.

\subsection{Error estimate for the 'ideal' method}

Before presenting the result, we recall that the quasi-interpolation operator $I_H$ (defined in (\ref{def-weighted-clement})) is locally stable and fulfills the typical approximation properties (c.f. \cite{MR1736895,MR1706735}), i.e. there exists a generic constant $C$, depending on the shape regularity of $\T_H$ but not on the local mesh size $H_T:=\operatorname{diam}(T)$, such that for all $v\in H^1(\Omega)$ and for all $T\in\T_H$ it holds
\begin{equation}\label{e:interr}
 H_T^{-1}\|v-\Ic v\|_{L^{2}(T)}+\|\nabla(v-\Ic v)\|_{L^{2}(T)}\leq C \| \nabla v\|_{L^2(\omega_T)}.
\end{equation}
Here, we denote $\omega_T:=\cup\{K\in\T_H\;\vert\;K\cap T\neq\emptyset\}$. The approximation and stability properties of the Cl\'ement-type quasi-interpolation operator were shown in \cite{MR1706735}, but only for triangular meshes. In \cite{MR1736895} they are also proved for quadrilateral meshes but in this latter work the weights $v_z$ in (\ref{def-weighted-clement}) are slightly modified to account for boundary corrections. However, from the proofs in \cite{MR1736895,MR1706735} it is clear, that estimate (\ref{e:interr}) (as it can be found in \cite{MR1706735}) directly generalizes to quadrilateral meshes.

The following theorem guarantees that, in the ideal (but impractical) case of no localization (i.e. full sampling $U(T)=\Omega$), the proposed LOD method preserves the common linear order convergence for the $H^1$-error without suffering from pre-asymptotic effects due to the rapid variations in $A$.

\begin{theorem}[A priori error estimate for $U(T)=\Omega$]
\label{convergence-theorem-max-undersampling}Assume (A1)-(A3) and $U(T)=\Omega$ for all $T\in\T_H$. If $u_h$ denotes the solution of the reference problem (\ref{reference-problem-eq}) and $u_{\LOD}$ the corresponding LOD approximation given by Definition \ref{definition-loc-lod-approx}, then it holds
\begin{align}
\label{convergence-equation-max-undersampling}\|  u_{\LOD} - u_h \|_{H^1(\Omega)} \lesssim H \| f \|_{L^2(\Omega)}.
\end{align}
\end{theorem}

\begin{proof}
Let $U(T)=\Omega$. Using (\ref{definition-of_V-H-c}) and the definition of the corrector operator $Q_h$ the $(A\nabla\cdot,\nabla\cdot)$-orthogonal complement of $W_h$ in $V_h$ is given by
\begin{align*}
V_H^c =(1-P_{A,h})(V_H) = (1+Q_h)(V_H) = \{ \Phi_H + Q_h(\Phi_H)| \hspace{2pt} \Phi_H \in V_H \}.
\end{align*}
With (\ref{reference-problem-eq}) and (\ref{lod-problem-eq}), we get for all $\Phi_H^c \in V_H^c$:
\begin{eqnarray*}
\lefteqn{\int_{\Omega} A \nabla \left( R_h(v_H) + Q_h(g_h) -  B_{h} \right) \cdot \nabla \Phi_H^c}\\
&\overset{(\ref{lod-problem-eq})}{=}& \int_{\Omega} f  \Phi_H^c - \int_{\Omega} A \nabla g_h \cdot \nabla \Phi_H^c +\int_{\Gamma_N} q \Phi_H^c \overset{(\ref{reference-problem-eq})}{=} \int_{\Omega} A \nabla v_h \cdot \nabla \Phi_H^c.
\end{eqnarray*}
Together with $V_h = V_H^c \oplus W_h$ and $V_H^c {\perp} W_h$ this implies $R_h(v_H) + Q_h(g_h) -  B_{h} - v_h \in W_h$ and therefore
\begin{align}
\label{error-in-remainder-space}I_H( R_h(v_H) + Q_h(g_h) -  B_{h} - v_h ) = 0.
\end{align}
Now, let $w_h \in W_h$ be arbitrary (which implies $I_H(w_h)=0$), we obtain
\begin{eqnarray*}
\lefteqn{\int_{\Omega} A \nabla \left( R_h(v_H) + Q_h(g_h) -  B_{h} - v_h \right) \cdot \nabla w_h} \\
&\overset{(\ref{reference-problem-eq})}{=}& \int_{\Omega} A \nabla \left( R_h(v_H) + R_h(g_h) \right) \cdot \nabla w_h  - \int_{\Omega} A \nabla  B_{h} \cdot \nabla w_h - \int_{\Omega} f w_h - \int_{\Gamma_N} q w_h \\
&\overset{(\ref{local-corrector-problem})}{=}& - \int_{\Omega} A \nabla B_{h} \cdot \nabla w_h - \int_{\Omega} f w_h - \int_{\Gamma_N} q w_h 
\overset{(\ref{local-boundary-corrector-problem})}{=} - \int_{\Omega} f w_h = \int_{\Omega} f ( I_H(w_h) - w_h ).
\end{eqnarray*}
Using (\ref{error-in-remainder-space}), we can choose $w_h = e_h := R_h(v_H) + Q_h(g_h) -  B_{h} - v_h$ to obtain:
\begin{eqnarray*}
\lefteqn{\| A^{1/2} \nabla \left( u_{\LOD} - u_h \right) \|_{L^2(\Omega)}^2
= \| A^{1/2} \nabla \left( R_h(v_H) + R_h(g_h) -  B_{h} - v_h - g_h\right) \|_{L^2(\Omega)}^2}\\
&=& \| A^{1/2} \nabla \left( R_h(v_H) + Q_h(g_h) -  B_{h} - v_h\right) \|_{L^2(\Omega)}^2\\
&=& \int_{\Omega} f ( I_H(e_h) - e_h ) \lesssim H \| f \|_{L^2(\Omega)} \|  \nabla e_h \|_{L^2(\Omega)} \\
&\lesssim& H \| f \|_{L^2(\Omega)} \| A^{1/2} \nabla \left( u_{\LOD} - u_h \right) \|_{L^2(\Omega)}.\hspace{150pt}
\end{eqnarray*}
\end{proof}

Assume again that $U(T)=\Omega$ for all $T\in \T_H$. Observe that by Theorem \ref{convergence-theorem-max-undersampling} we get
\begin{align}
\nonumber\| \nabla \left( R_h(v_H) + Q_h(g_h) -  B_{h} \right) \|_{L^2(\Omega)} &\le \| \nabla \left( u_{\LOD} - u_h\right)  \|_{L^2(\Omega)} +  \| \nabla u_h  \|_{L^2(\Omega)}\\
\label{energy-estimate-ideal}&\lesssim \|f\|_{L^2(\Omega)} + \|\nabla g_h \|_{L^2(\Omega)} + \| q \|_{L^2(\Gamma_N)},
\end{align}
with a constant independent of the variations in the data. By using the stability (\ref{e:interr}) of the quasi-interpolation operator $I_H$ the above estimate implies
\begin{eqnarray}
\nonumber
\| \nabla v_H \|_{L^2(\Omega)} &=& 
\| \nabla I_H \left( R_h(v_H) + Q_h(g_h) -  B_{h} \right) \|_{L^2(\Omega)}\\
\label{energy-estimate-ideal-2}&\overset{(\ref{energy-estimate-ideal})}{\lesssim}& \|f\|_{L^2(\Omega)} + \|\nabla g_h \|_{L^2(\Omega)} + \| q \|_{L^2(\Gamma_N)}.
\end{eqnarray}

\subsection{Error estimates for the localized method}

Theorem \ref{convergence-theorem-max-undersampling} gave us a first hint that the method is capable of preserving the usual convergence rates. However, the case of full sampling (i.e. $U(T)=\Omega$) is not computationally feasible, since the cost for solving one corrector problem would be identical to the cost of solving the original problem on the full fine scale. The key issue is therefore to find a 'minimum size' for the localization patches $U(T)$, so that we still preserve the rate obtained in Theorem \ref{convergence-theorem-max-undersampling}. Let us first specify what we understand by the notion 'patch size'.

\begin{definition}[Patch size]
\label{category-k}Let $U(T)$ be an admissible patch and let $x_{U(T)}\in U(T)$ denote the barycenter of the patch. We say that $U(T)$ is of category $m\in\mathbb{N}$ if
\begin{align*}
|x_{U(T)} - \bar{x} | \ge m |\log(H)| H \quad \mbox{for all } \bar{x} \in \partial U(T) \setminus \partial \Omega.
\end{align*}
\end{definition}
If $U(T)\cap \partial \Omega = \emptyset$, a category $m$ patch is nothing but a patch with diameter $ 2m |\log(H)| H$. The generalized definition above accounts for the fact that we know the correct boundary condition on $\partial \Omega$ and that we do not have to deal with a decay behavior there.

The following abstract lemma shows that any solution of a generalized corrector problem (with respect to $T\in \T_H$) exponentially decays to zero outside $T$. In order to quantify the decay with respect to the coarse grid, we introduce patches $U(T)$ that consist of $k$ coarse element layers attached to $T$ (i.e. $U(T)$ is a category $m=\lfloor k/|\log(H)|\rfloor$ patch). 

\begin{lemma}[Decay of local correctors]
\label{lemma-influence-intersections}
Let $k\in \mathbb{N}_{>0}$ be fixed. We define patches where the extension layer consists of a fixed number of coarse element layers. For all $T\in\T_H $, we define element patches in the coarse mesh $\T_H$ by
\begin{equation}\label{def-patch-U-k}
    \begin{aligned}
      U_0(T) & := T, \\
      U_k(T) & := \cup\{T'\in \T_H\;\vert\; T'\cap U_{k-1}(T)\neq\emptyset\}\quad k=1,2,\ldots .
    \end{aligned}
\end{equation}
Now, let $p_h^T \in W_h$ be the solution of
\begin{align}
\label{generalized-corrector-problem-2}\int_{\Omega} A \nabla p_h^T \cdot \nabla \phi_h =F_T(\phi_h) \qquad \mbox{for all } \phi_h \in W_h
\end{align}
where $F_T\in W_h^{\prime}$ is such that $F_T(\phi_h)=0$ for all $\phi_h \in \mathring{W}_h(\Omega \setminus T)$. Furthermore, we let $p_h^{T,k} \in \mathring{W}_h(U_k(T))$ denote the solution of
\begin{align}
\label{generalized-corrector-problem-localized}\int_{U_k(T)} A \nabla p_h^{T,k} \cdot \nabla \phi_h =F_T(\phi_h) \qquad \mbox{for all } \phi_h \in \mathring{W}_h(U_k(T)).
\end{align}
Then there exists a generic constant $0<\theta<1$ that depends on the contrast but not on $H$, $h$ or the variations of $A$ such that
\begin{eqnarray}
\label{equation-influence-intersections}\left\| \sum_{T\in\T_H} \nabla (p_h^T-p_h^{T,k})\right\|_{L^2(\Omega)}^2 \lesssim k^d
\theta^{2 k}
\sum_{T\in\T_H} \|\nabla p_h^T \|_{L^2(\Omega)}^2.
\end{eqnarray}
\end{lemma}
The proof of Lemma \ref{lemma-influence-intersections} is postponed to the appendix. It is similar to the proofs given in \cite{MaPe12} and \cite{Henning:Peterseim:2012}, but with some technical details that account for the boundary conditions and the possibly quadrilateral partition of $\Omega$. Using Lemma \ref{lemma-influence-intersections} we can quantify what is a sufficient size of the localization patches $U(T)$:

\begin{theorem}[A priori error estimates for the localized method]\label{t:a-priori-local}$\\$
Assume (A1)-(A3). Given $k\in\mathbb{N}_{>0}$, let $U(T)=U_k(T)$ for all $T \in\T_H$ where $U_k(T)$ is defined as in Lemma \ref{lemma-influence-intersections}. By $u_h$ we denote the solution of the reference problem (\ref{reference-problem-eq}) and by $u_{\LOD}$ we denote the LOD approximation introduced in Definition \ref{definition-loc-lod-approx}. Then, the following a priori error estimates hold true
\begin{align*}
 \|\nabla u_h-\nabla u_{\LOD}\|_{L^2(\Omega)}&\lesssim (H + k^{\frac{d}{2}} \theta^{k} ) \|f\|_{L^2(\Omega)} + k^{\frac{d}{2}} \theta^{k} ( \|\nabla g_h \|_{L^2(\Omega)} + \| q \|_{L^2(\Gamma_N)}),\\
 \| u_h- u_{\LOD}\|_{L^2(\Omega)}&\lesssim (H + k^{\frac{d}{2}} \theta^{k} ) \|\nabla u_h-\nabla u_{\LOD}\|_{L^2(\Omega)},
\end{align*}
where $0<\theta<1$ is as in Lemma \ref{lemma-influence-intersections}.
\end{theorem}

\begin{remark}[Discussion of localization strategies]
Assume that $\Gamma_D=\partial \Omega$ and that $g=0$. The LOD is based on an appropriate localization of the optimal correction operator $\mathbf{Q}_h: V_H \rightarrow W_h$ given by (\ref{optimal-global-corrector}). Furthermore, $k>0$ is an integer. 

In \cite{MaPe12} it was proposed to pick a $k$-layer environment $U_k(\omega_z)$ of $\omega_z:=\mbox{supp}(\Phi_z)$ for every coarse nodal basis function $\Phi_z$ ($z \in \mathcal{\mathcal{N}}_H$) and to solve for $\lambda_z \in \mathring{W}_h(U_k(\omega_z))$ with
\begin{align*}
\int_{U(\omega_z)} A \nabla \lambda_z  \cdot \nabla w_h = - \int_{\omega_z} A \nabla \Phi_z \cdot \nabla w_h \qquad \mbox{for all } w_h \in \mathring{W}_h(U_k(\omega_z)).
\end{align*}
For arbitrary $\Phi_H \in V_H$, the approximation of the optimal global corrector $\mathbf{Q}_h$ is then given by $Q_h^{\mathbf{1}}(\Phi_H):= \sum_{z \in \mathcal{N}_H} \Phi_H(z) (\lambda_z + \Phi_z)$. Since '$\lambda_z + \Phi_z$' does not form a partition of unity, the localization error is polluted by the factor $(1/H)$, i.e. we obtain the worse estimate
\begin{align*}
 \|\nabla u_h-\nabla u^{ms}\|_{L^2(\Omega)}&\lesssim (H + (1/H) k^{\frac{d}{2}} \theta^{k} ) \|f\|_{L^2(\Omega)}.
\end{align*}
The factor $(1/H)$ can be numerically observed and leeds to larger patches $U_k(\omega_z)$.

In \cite{Henning:Peterseim:2012} it was proposed to solve for $w_{h,T,i} \in \mathring{W}_h(U(T))$ (for $1\le i \le d$) with
\begin{align*}
\int_{U(T)} A \nabla w_{h,T,i} \cdot \nabla w_h = - \int_{T} A e_i \cdot \nabla w_h \qquad \mbox{for all } w_h \in \mathring{W}_h(U(T)),
\end{align*}
where $e_i \in \R^d$ denotes the $i$'th unit vector in $\R^d$ (i.e. $(e_i)_j = \delta_{ij}$). The approximation of the global corrector is given by $Q_h^{\mathbf{2}}(\Phi_H):= \sum_{T\in \mathcal{T}_H} \sum_{i=1}^d \partial_{x_i} \Phi_H(x_T) w_{h,T,i}$, where $x_T$ denotes the barycenter of $T$. This approach is motivated from homogenization theory and leads to the same error estimates as presented in Theorem \ref{t:a-priori-local}. However, this strategy is restricted to P1 Finite Elements on triangular grids (in this case it is equivalent to the strategy presented in this paper) and in particular it fails for quadrilateral grids.

Another localization strategy, also based on a partition of unity for the right hand side of the local problems, was proposed in \cite{Larson:Malqvist:2007}. Similar a-priori error estimates can be expected, however, in the mentioned work, more local problems need to be solved.
\end{remark}

\begin{conclusion}
\label{conclusion-convergence-rates}
Let assumptions (A1)-(A3) be fulfilled, let $\T_H$ be a given coarse triangulation and let $\mathcal{U}$ denote a corresponding set of admissible patches, with the property that each patch $U(T)$ is of category $m\in \mathbb{N}_{>0}$ (in the sense of Definition \ref{category-k}). Then for arbitrary mesh sizes $H\ge h$ it holds
\begin{align*}
 \|\nabla u_h-\nabla u_{\LOD}&\|\lesssim H \|f\|_{L^2(\Omega)} + H^m (\|\nabla g_h \|_{L^2(\Omega)} + \| q \|_{L^2(\Gamma_N)}),\\
 \| u_h- u_{\LOD}&\|\lesssim H^2 \|f\|_{L^2(\Omega)} + H^{2m} ( \|\nabla g_h \|_{L^2(\Omega)} + \| q \|_{L^2(\Gamma_N)}).
\end{align*}
Observe that powers in $m$ are obtained from Theorem \ref{t:a-priori-local} by choosing $k \gtrsim m \log(H^{-1})$.
\end{conclusion}

\begin{proof}[Proof of Theorem~\ref{t:a-priori-local}]
Let $Q_h^T: V_h \rightarrow \mathring{W}_h(U_k(T))$ denote the correction operator defined according to (\ref{local-corrector-problem}) and let $Q_h^{\Omega,T}: V_h \rightarrow W_h$ denote the 'ideal' correction operator for $U(T)=\Omega$. Likewise, by $B_{h}^{T} \in \mathring{W}_h(U(T))$ we denote the boundary corrector given by (\ref{local-boundary-corrector-problem}) and by $B_{h}^{\Omega,T} \in W_h$ we denote the solution of (\ref{local-boundary-corrector-problem}) for $U(T)=\Omega$. In the same way, we distinguish between $Q_h$ and $Q_h^{\Omega}$; $B_h$ and $B_h^{\Omega}$ and $u_{\LOD}$ and $u_{\LOD}^{\Omega}$. The coarse part $v_H$ of the LOD approximation is defined by (\ref{lod-problem-eq}) for $U(T)=U_k(T)$ and by $v_H^{\Omega}$ for $U(T)=\Omega$. Let $\Phi_H \in V_H$ be arbitrary. Using the Galerkin orthogonality
\begin{eqnarray}
\label{galerkin-orthogonality-error-equation}
\int_{\Omega} A \nabla \left( R_h(v_H) + Q_h(g_h) -  B_{h} -v_h \right) \cdot \nabla R_h(\Phi_H) = 0,
\end{eqnarray}
we get
\begin{eqnarray}
\label{galerkin-orthogonality-in-final-proof}\nonumber\lefteqn{\|A^{1/2} \nabla \left( R_h(v_H) + Q_h(g_h) -  B_{h} -v_h \right)\|_{L^2(\Omega)}}\\
&\leq& \|A^{1/2}\nabla \left( R_h(\Phi_H) + Q_h(g_h) -  B_{h} -v_h \right)\|_{L^2(\Omega)}.
\end{eqnarray}
This yields
\begin{eqnarray*}
\lefteqn{\|\nabla u_h-\nabla u_{\LOD}\|_{L^2(\Omega)}}\\
&=& \|\nabla v_h -\nabla R_h(v_H) - \nabla Q_h(g_h) + \nabla  B_{h} \|_{L^2(\Omega)}\\
&\overset{(\ref{galerkin-orthogonality-in-final-proof})}{\lesssim}& \|\nabla v_h-\nabla v_H^{\Omega} - \nabla Q_h(v_H^{\Omega}) - \nabla Q_h(g_h) + \nabla  B_{h}\|_{L^2(\Omega)}\\
&\le& 
\|\nabla v_h +  \nabla g_h -\nabla v_H^{\Omega} - \nabla Q_h^{\Omega}(v_H^{\Omega}) - \nabla Q_h^{\Omega}(g_h) + \nabla  B_{h}^{\Omega} -  \nabla g_h\|_{L^2(\Omega)}\\
&\enspace&+  \|\nabla \hspace{-3pt}\left( Q_h- Q_h^{\Omega}\right)\hspace{-3pt}(v_H^{\Omega})\|_{L^2(\Omega)}
+ \|\nabla \hspace{-3pt}\left( Q_h- Q_h^{\Omega}\right)\hspace{-3pt}(g_h)\|_{L^2(\Omega)}
+  \|\nabla \hspace{-3pt}\left( B_h- B_h^{\Omega}\right)\|_{L^2(\Omega)}
 \\
&\overset{(\ref{equation-influence-intersections})}{\lesssim}& \|\nabla u_h-\nabla u_{\LOD}^{\Omega}\|_{L^2(\Omega)}\\
&\enspace&+ k^{d/2} \theta^{ k} \hspace{-2pt} \left( \sum_{T\in\T_H} \|\nabla Q_h^{\Omega,T}(v_H^{\Omega}) \|_{L^2(\Omega)}^2 + \|\nabla Q_h^{\Omega,T}(g_h) \|_{L^2(\Omega)}^2 + \|\nabla B_h^{\Omega,T} \|_{L^2(\Omega)}^2  \right)^{1/2} \hspace{-4pt}.
\end{eqnarray*}
Equation \eqref{convergence-equation-max-undersampling} and the estimates
\begin{align*}
\sum_{T\in\T_H}  \|\nabla Q_h^{\Omega,T}(v_H^{\Omega}) \|_{L^2(\Omega)}^2 &\overset{(\ref{local-corrector-problem})}{\lesssim} \sum_{T\in\T_H}  \|\nabla v_H^{\Omega} \|_{L^2(T)}^2 =  \|\nabla v_H^{\Omega} \|_{L^2(\Omega)}^2 \\
&\overset{(\ref{energy-estimate-ideal-2})}{\lesssim} \|f\|_{L^2(\Omega)}^2 + \|\nabla g_h \|_{L^2(\Omega)}^2 + \| q \|_{L^2(\Gamma_N)}^2,
\end{align*}
and
\begin{align*}
\|\nabla Q_h^{\Omega,T}(g_h) \|_{L^2(\Omega)} &\overset{(\ref{local-corrector-problem})}{\lesssim} \|\nabla g_h \|_{L^2(T)} \quad \mbox{and} \quad
\|\nabla B_h^{\Omega,T} \|_{L^2(\Omega)} \overset{(\ref{local-boundary-corrector-problem})}{\lesssim} \|q \|_{L^2(T\cap \Gamma_N)},
\end{align*}
readily yield the assertion for the $H^1$-error. The $L^2$-error estimate is obtained by a Aubin-Nitsche duality argument. We define $e_h:=u_h - u_{\LOD}$. Note that $e_h \in V_h$, but in general not in $W_h$ (only for $U(T)=\Omega$). We consider two dual problems (that correspond to problems with homogenous Dirichlet and Neumann boundary condition): find $z_h \in V_h$ with
\begin{align}
\label{dual-problem-fine-space}\int_{\Omega} A \nabla \phi_h \cdot \nabla z_h = \int_{\Omega} e_h \phi_h \quad \mbox{for all } \phi_h \in V_h
\end{align}
and find $z_H \in V_H$ with
\begin{align}
\label{dual-problem-ms-space}\int_{\Omega} A \nabla R_h(\Phi_H) \cdot \nabla R_h(z_H) = \int_{\Omega} e_h R_h(\Phi_H) \quad \mbox{for all } \Phi_H \in V_H.
\end{align}
As in the previous case, we get
\begin{align}
\label{dual-estimate}\| \nabla (z_h - R_h(z_H)) \|_{L^2(\Omega)} \lesssim (H + k^{d/2} \theta^{ k}) \| e_h \|_{L^2(\Omega)}.
\end{align}
On the other hand we have with $e_h \in V_h$
\begin{align*}
\| e_h \|_{L^2(\Omega)}^2 &\overset{(\ref{dual-problem-fine-space})}{=} \int_{\Omega} A \nabla e_h \cdot \nabla z_h  \overset{(\ref{galerkin-orthogonality-error-equation})}{=}  \int_{\Omega} A \nabla e_h \cdot (\nabla z_h - \nabla R_h(z_H) )\\
& \overset{(\ref{dual-estimate})}{\lesssim} \| \nabla e_h \|_{L^2(\Omega)} ( H + k^{d/2} \theta^{ k}) \| e \|_{L^2(\Omega)}.
\end{align*}
Dividing by $\| e_h \|_{L^2(\Omega)}$ and with the previously derived estimate for $\| \nabla e_h \|_{L^2(\Omega)}$ we obtain the $L^2$ error estimate.
\end{proof}

\section{Numerical experiments}
\label{section-numerical-experiments}

In this section we present three different model problems with corresponding numerical results. The first model problem is to demonstrate the usability of the boundary correctors. Here we prescribe a Dirichlet boundary condition that is rapidly oscillating and that cannot be captured by the coarse grid. However, we will see that the Dirichlet boundary correctors perfectly capture its effect. In the second numerical experiment we investigate the influence of a very thin isolator close to the boundary of the domain in the case of a non-zero Dirichlet boundary condition. This leads to a solution with very narrow accumulations that cannot be described on the coarse scale, but which are accurately resembled by the LOD approximation. In the third model problem we investigate how the method reacts to channels of high conductivity and an additional isolator channel. These channels are very thin and long. Typically, such channels have to be either resolved by the coarse mesh or the localized patches must be large enough so that each channel is contained in a patch. In our experiment we observe that neither is necessary if we apply the LOD to this model problem.

\begin{figure}[h]
\centering
\includegraphics[width=0.8\textwidth]{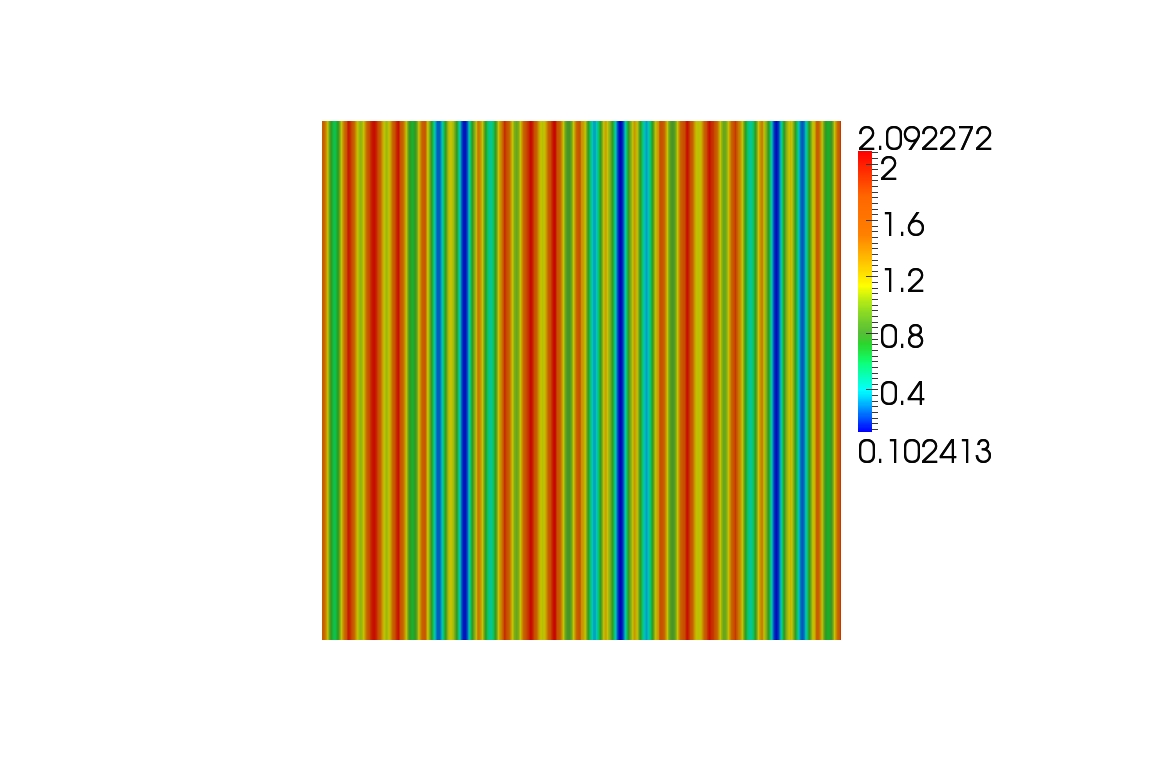}
\caption{\it Model problem 1. Plot of the rapidly varying diffusion coefficient $A$ given by equation (\ref{eq:coefficient_problem_1}). It takes values between about $0.1$ and $2.1$.}
\label{diffusion_problem_1}
\end{figure}

\begin{figure}[h]
\centering
\includegraphics[width=1.0\textwidth]{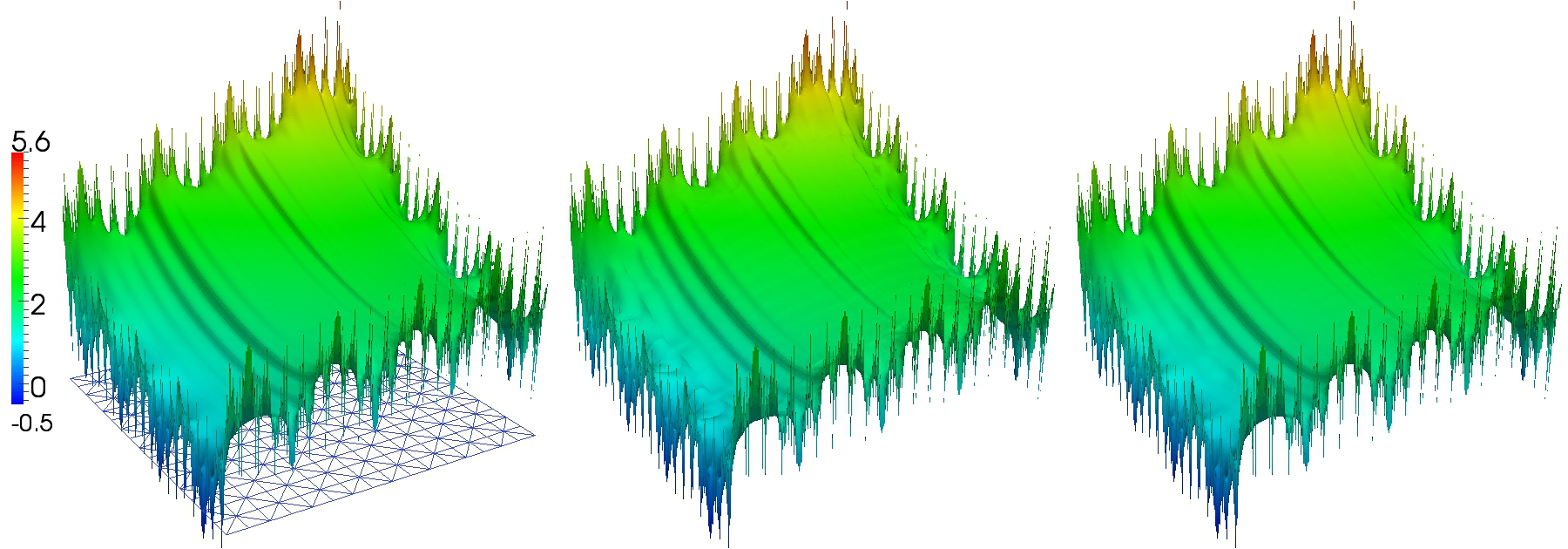}
\caption{\it Model problem 1. Computations made for $H=2^{-4}$ and $h=2^{-8}$. The left picture shows the standard FEM reference solution on the fine grid and, below, the coarse grid for comparison. The middle picture shows the LOD approximation obtained for $1$ coarse grid layer and the right picture shows the LOD approximation for $2$ coarse grid layers.}
\label{mp-1-serie}
\end{figure}

\begin{figure}[h]
\centering
\includegraphics[width=0.8\textwidth]{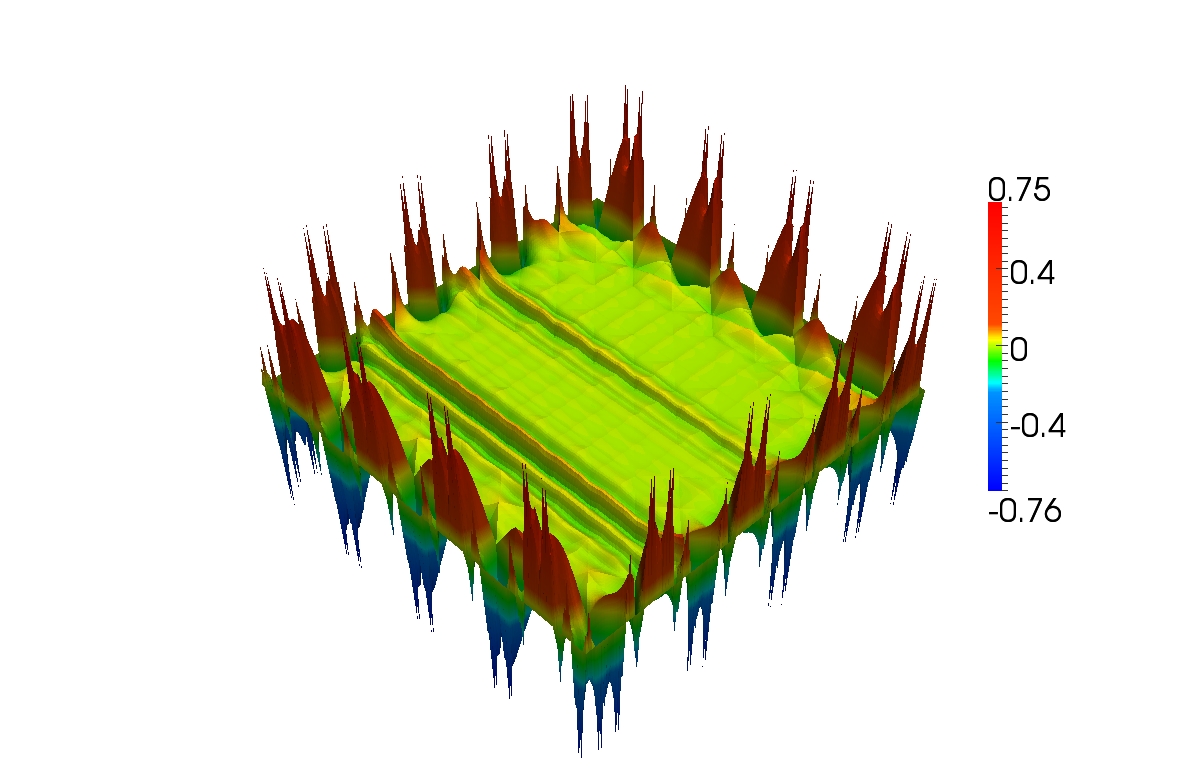}
\caption{\it Model problem 1. Computations made for $H=2^{-4}$, $h=2^{-8}$ and $k=2$ coarse grid layers around each $T\in\T_H$ for localization. The picture shows the fine part (i.e. corrector part $Q_h(v_H+g_h)-B_h$) of the LOD approximation.}
\label{mp-1-fine-part}
\end{figure}

\begin{table}[h]\footnotesize
\caption{\it Model problem 1. Computations made for $h=2^{-8}$, i.e. $|\mathcal{T}_h|=131072$ and $|\mathcal{N}_h|=66049$. $k$ denotes the number of coarse layers. $|\mathcal{T}_h^\mathcal{U}|$ and $|\mathcal{N}_h^\mathcal{U}|$, the averages for elements and nodes in a patches, are defined in (\ref{def-average-nodes-elements}). The table depicts errors between $u_h$ and $u_{\LOD}$.}
\label{serie-convergence-mp-1}
\begin{center}
\begin{tabular}{|c|c|c|c|c|c|c|}
\hline $H$      & \mbox{k} & $\|u_h - u_{\LOD}\|_{L^2(\Omega)}^{\mbox{\tiny rel}}$ & $\|u_h - u_{\LOD}\|_{H^1(\Omega)}^{\mbox{\tiny rel}}$ & $|\mathcal{T}_h^\mathcal{U}|$ & $|\mathcal{N}_h^\mathcal{U}|$ \\
\hline
\hline $2^{-2}$ & 0.5 & 0.03593 & 0.07684 & 22480 & 11465 \\
\hline $2^{-3}$ & 1   &  0.00824 & 0.04241 & 14696 & 7525\\
\hline $2^{-4}$ & 2   & 0.00162 & 0.01664 & 10743 & 5520 \\
\hline $2^{-5}$ & 4   & 0.00024 & 0.00453 & 8922 & 4596 \\
\hline
\end{tabular}
\end{center}
\end{table}

\begin{table}[h]\footnotesize
\caption{\it Model problem 1. Computations made for $H=2^{-4}$ and $h=2^{-8}$, i.e. $|\mathcal{T}_h|=131072$ and $|\mathcal{N}_h|=66049$. In the first column, the number of fine grid element layers is shown, $k$ denotes the corresponding number of coarse grid element layers. $|\mathcal{T}_h^\mathcal{U}|$ and $|\mathcal{N}_h^\mathcal{U}|$ are defined in (\ref{def-average-nodes-elements}). The table depicts $L^2$- and $H^1$-errors.}
\label{serie-decay-mp-1}
\begin{center}
\begin{tabular}{|c|c|c|c|c|c|c|}
\hline \mbox{Fine layers} &
          \mbox{k} & $\|u_h - u_{\LOD}\|_{L^2(\Omega)}^{\mbox{\tiny rel}}$ & $\|u_h - u_{\LOD}\|_{H^1(\Omega)}^{\mbox{\tiny rel}}$
          & $|\mathcal{T}_h^\mathcal{U}|$ & $|\mathcal{N}_h^\mathcal{U}|$ \\
\hline
\hline 4 & 0.25    &  0.02699 & 0.24344 & 847     & 471    \\
\hline 8 & 0.5      &  0.01593 & 0.14345 & 1675   & 900    \\
\hline 16 & 1       &  0.00508 & 0.05071 & 3994   & 2090  \\
\hline 32 & 2       &  0.00162 & 0.01664 & 10743 & 5520  \\
\hline 64 & 4       &  0.00017 & 0.00185 & 30599 & 15548 \\
\hline
\end{tabular}
\end{center}
\end{table}

In this section, we let $u_h$ denote the fine scale reference given by (\ref{reference-problem-eq}) and we let $u_{LOD}$ denote the LOD approximation given by Definition \ref{definition-loc-lod-approx}. All errors are relative errors denoted by
\begin{align*}
\|u_h - u_{\LOD}\|_{L^2(\Omega)}^{\mbox{\tiny rel}}&:=\frac{\|u_h - u_{\LOD}\|_{L^2(\Omega)}}{\|u_h\|_{L^2(\Omega)}} \quad \mbox{and} \\ 
\|u_h - u_{\LOD}\|_{H^1(\Omega)}^{\mbox{\tiny rel}}&:=\frac{\|u_h - u_{\LOD}\|_{H^1(\Omega)}}{\|u_h\|_{H^1(\Omega)}}.
\end{align*}
In the following, we use localization patches that we construct by adding {\it fine grid element layers} to a {\it coarse grid element}, i.e. for a given fixed number of fine layers $\ell\in \mathbb{N}_{>0}$ and for $T\in\T_H $, we define element patches by
\begin{align*}
U_{h,0}(T) & := T \quad \mbox{and} \quad  U_{h,\ell}(T) := \cup\{T'\in \T_h\;\vert\; T'\cap U_{h,\ell-1}(T)\neq\emptyset\}\quad \ell=1,2,\ldots .
\end{align*}
This choice is more flexible than using full coarse grid element layers for constructing the patches. Still, in the spirit of definition (\ref{def-patch-U-k}), any number of fine grid element layers translates into a corresponding number of coarse grid element layers (which might be fractional then). For the readers convenience we will state both numbers, even though they can be concluded from each other. Subsequently, $\lfloor \cdot \rfloor$ denotes the floor function. For fixed $\T_h$ and fixed set of patches $\mathcal{U}$ (see Definition \ref{def-admissible-patches}) we denote by $|\mathcal{T}_h^\mathcal{U}|$ and $|\mathcal{N}_h^\mathcal{U}|$ the average number of elements and the average number of nodes in the patches, i.e.
\begin{align}
\label{def-average-nodes-elements}
|\mathcal{T}_h^\mathcal{U}| := |\mathcal{U}|^{-1} \sum_{U \in \mathcal{U}} |\mathcal{T}_h(U)|
\qquad
\mbox{and}
\qquad
|\mathcal{N}_h^\mathcal{U}| := |\mathcal{U}|^{-1} \sum_{U \in \mathcal{U}} |\mathcal{N}_h(U)|.
\end{align}

\subsection{Model problem 1}

We consider the following model problem.
\begin{problem}
Let $\Omega:= ]0,1[^2$ and $\epsilon:=0.05$. Find $u \in H^1(\Omega)$ such that
\begin{equation*}\label{eq:model-1}
  \begin{aligned}
    -\nabla \cdot A(x)\nabla u(x) &= 1\quad \hspace{181pt}\text{\rm in }\Omega, \\
    u(x) & = \sin \left( \frac{2\pi}{\epsilon} x_1 \right) +\cos \left( \frac{2\pi}{\epsilon} x_2 \right) + \frac{1}{2} e^{x_1 + x_2} 
     \quad \hspace{7pt}\text{\rm on }\partial \Omega,
  \end{aligned}
\end{equation*}
where
\begin{align}
\label{eq:coefficient_problem_1} A(x_1,x_2):=\frac{11}{10} + \frac{1}{2} \sin \left( \left \lfloor \frac{x_1}{\epsilon} \right \rfloor \right) +  \frac{1}{2} \cos \left( 2\pi \frac{x_1}{\epsilon} \right).
\end{align}
$A$ is depicted in Figure \ref{diffusion_problem_1}.
\end{problem}
This first model problem involves a Dirichlet boundary condition that is rapidly oscillating and that cannot be accurately described on the coarse scale. We want to investigate how the Dirichlet boundary corrector captures these effects to incorporate them in the final LOD approximation without resolving the boundary with the coarse mesh. The reference solution was obtained with a standard finite element method for $h=2^{-8}$. First, we choose the coarse grid with mesh size $H$ such that $h=H^2$. In Figure \ref{mp-1-serie} we can see the corresponding results. The left plot shows the reference solution, the middle plot the LOD approximation obtained using localized patches with $1$ coarse grid layer (in the sense of (\ref{def-patch-U-k})) and the right plot shows the LOD approximation with $2$ coarse grid layers. We observe that the boundary oscillations and all relevant fine scale features are perfectly captured by the LOD, even for small patch sizes and without resolving the boundary conditions with the coarse mesh. For $2$ coarse grid layers almost no difference to the reference solution is visible. The influence of the boundary corrector can be concluded from Figure \ref{mp-1-fine-part} where the whole fine scale part of $u_{\LOD}$ is depicted. We see that the boundary correctors contribute essential information. A quantitative comparison between reference solution and LOD is given in Tables \ref{serie-convergence-mp-1} and \ref{serie-decay-mp-1}. Table \ref{serie-convergence-mp-1} shows the error behavior if we double the number of coarse layers with each uniform coarse grid refinement (starting with half a coarse layer for $H=2^{-2}$). We observe up to quadratic convergence for the $H^1$-error and up to almost cubic convergence for $L^2$-error. Note that these high rates are only due to the doubling of the number of coarse layers, instead of increasing the patch thickness by the logarithmic factor $\log(H^{-1})$. We refer to the numerical experiments in \cite{Henning:Malqvist:Peterseim:2012} for detailed results on how the rates stated in Conclusion \ref{conclusion-convergence-rates} can be obtained by the logarithmic scaling. In Table \ref{serie-decay-mp-1}, the exponentially fast decay of the error with respect to coarse grid layers is demonstrated. Using the newly introduced boundary correctors, the LOD is able to accurately handle the rapidly varying Dirichlet boundary condition (in addition to the variations produced by the diffusion coefficient $A$).

\subsection{Model problem 2}

\begin{figure}[t]
\centering
\includegraphics[width=0.8\textwidth]{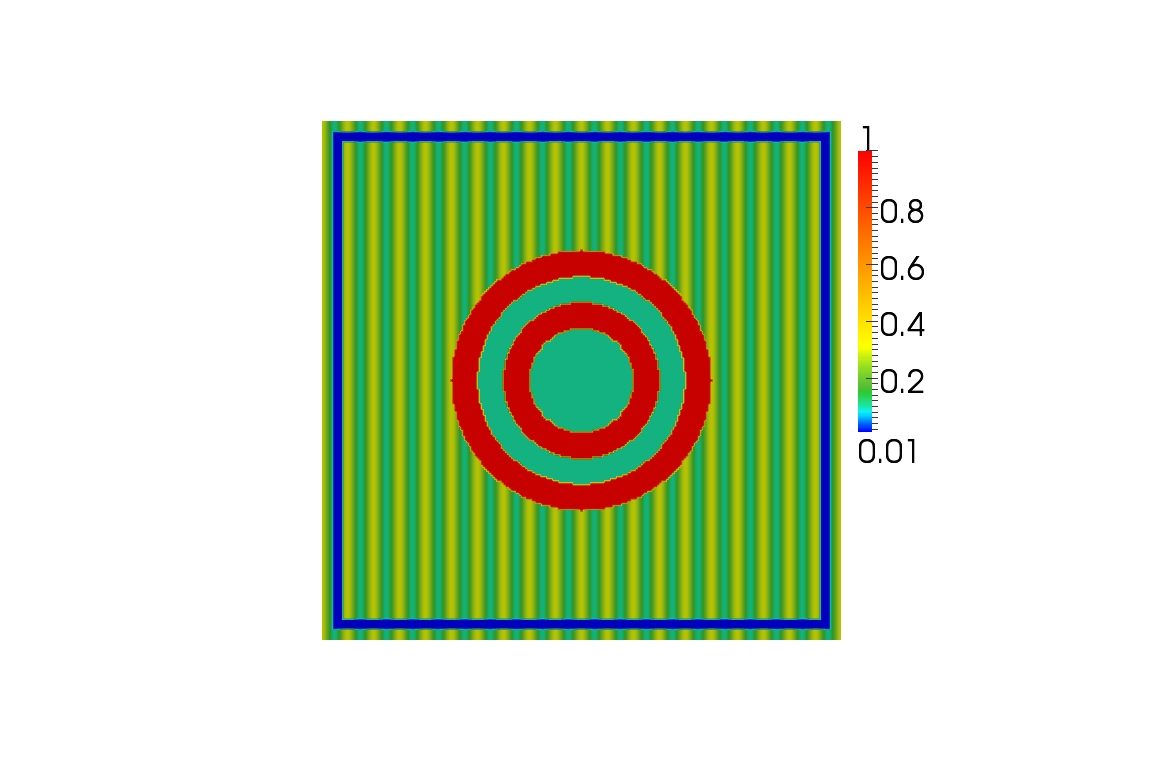}
\caption{\it Model problem 2. Plot of the diffusion coefficient $A$. It consists of a rapidly varying basis structure (green/yellow region) given by $\frac{1}{10} ( 2 + \cos ( 2 \pi \frac{x_1}{\epsilon}))$ for $\epsilon=0.05$. Here, $A$ takes values between $0.1$ and $0.3$. This structure is perturbed by an isolator that is located close to the boundary (blue region) and that has a conductivity of $0.01$. A second perturbation can be found in a ball of radius $0.25$ around the center of the domain. Here, the diffusivity changes its values in circular layers between $1$ (red region) and $0.1$ (turquoise region).}
\label{diffusion_problem_2}
\end{figure}

\begin{table}[t]\footnotesize
\caption{\it Model problem 2. Computations made for $h=2^{-8}$, i.e. $|\mathcal{T}_h|=131072$ and $|\mathcal{N}_h|=66049$. In the second column, the number of fine grid element layers is shown, $k$ denotes the corresponding number of coarse grid element layers. $|\mathcal{T}_h^\mathcal{U}|$ and $|\mathcal{N}_h^\mathcal{U}|$ are defined in (\ref{def-average-nodes-elements}). The table depicts $L^2$- and $H^1$-errors between $u_h$ and $u_{\LOD}$.}
\label{serie-convergence-mp-2}
\begin{center}
\begin{tabular}{|c|c|c|c|c|c|c|}
\hline $H$     & \mbox{Fine layers} & \mbox{k} & $\|u_h - u_{\LOD}\|_{L^2(\Omega)}^{\mbox{\tiny rel}}$ & $\|u_h - u_{\LOD}\|_{H^1(\Omega)}^{\mbox{\tiny rel}}$ & $|\mathcal{T}_h^\mathcal{U}|$ & $|\mathcal{N}_h^\mathcal{U}|$ \\
\hline
\hline $2^{-3}$ & 4   & 0.125 & 0.09234 & 0.50579 & 2047   & 1102 \\
\hline $2^{-3}$ & 8   & 0.25   & 0.06929 & 0.38912 & 3290   & 1738 \\
\hline $2^{-3}$ & 16 & 0.5     & 0.04636 & 0.26852 & 6340   & 3291 \\
\hline $2^{-3}$ & 32 & 1        & 0.01708 & 0.12064 & 14696 & 7525 \\
\hline $2^{-3}$ & 64 & 2        & 0.00655 & 0.07400 & 36398 & 18472 \\
\hline $2^{-3}$ & 96 & 3        & 0.00557 & 0.06996 & 61556 & 31131 \\
\hline
\hline $2^{-4}$ & 4   & 0.25  & 0.05513 & 0.35118 & 847     & 471 \\
\hline $2^{-4}$ & 8   & 0.5    & 0.02893 & 0.19508 & 1675   & 900 \\
\hline $2^{-4}$ & 16 & 1       & 0.00908 & 0.09389 & 3994   & 2090 \\
\hline $2^{-4}$ & 32 & 2       & 0.00159 & 0.03066 & 10743 & 5520 \\
\hline $2^{-4}$ & 48 & 3       & 0.00091 & 0.02269 & 19821 & 10111\\
\hline $2^{-4}$ & 64 & 4       & 0.00074 & 0.02011 & 30599 & 15548 \\
\hline
\end{tabular}
\end{center}
\end{table}

\begin{figure}[t]
\centering
\includegraphics[width=1.0\textwidth]{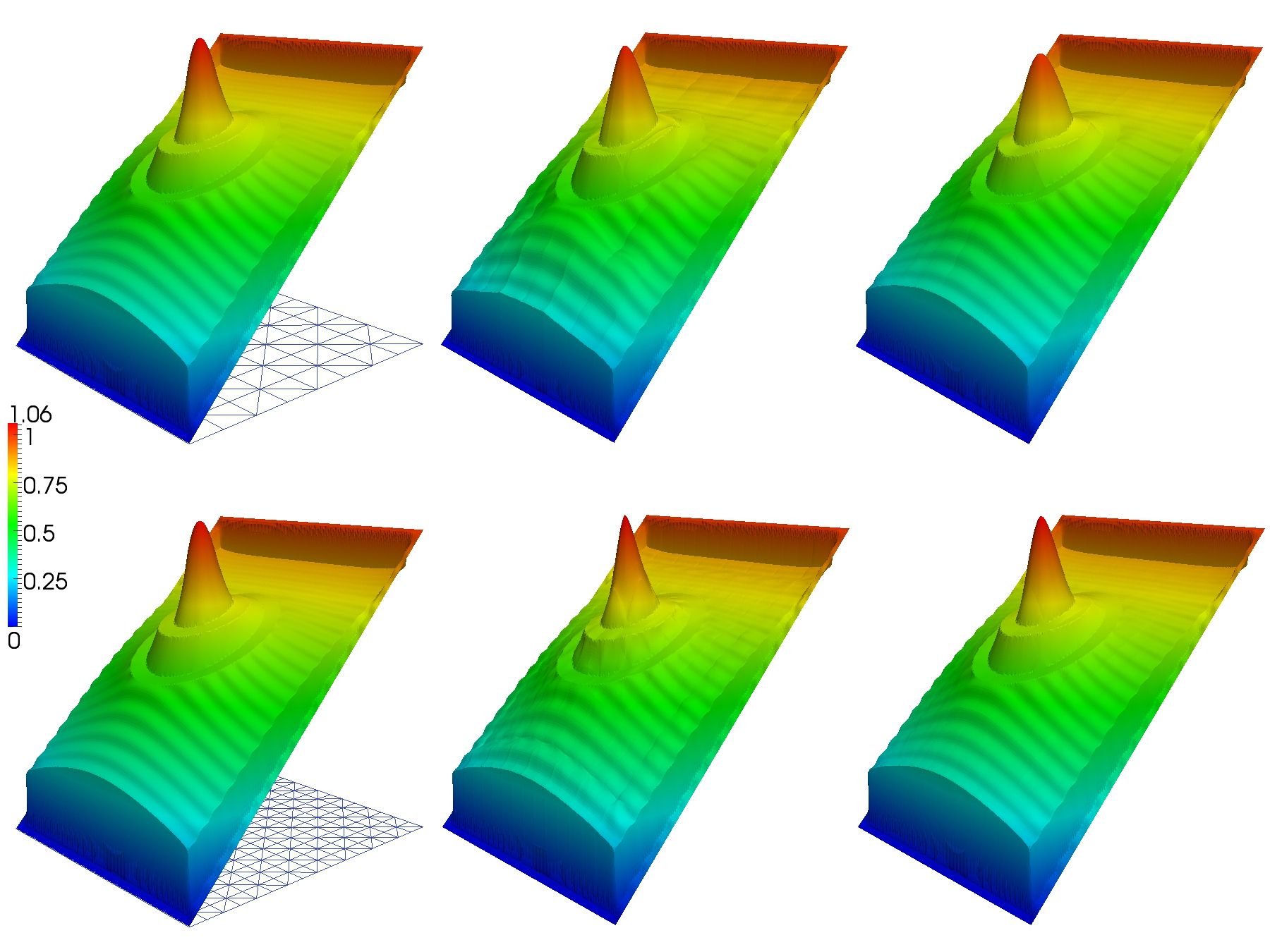}
\caption{\it Model problem 2. Computations made for $h=2^{-8}$ and respectively $H=2^{-3}$ in the upper row and $H=2^{-4}$ in the lower row. The left picture always shows the standard FEM reference solution on the fine grid (i.e. $h=2^{-8}$) and, below, the coarse grid for comparison ($H=2^{-3}$ and $H=2^{-4}$ respectively). The middle picture shows the LOD approximation obtained for $1$ coarse grid layer and the right picture shows the LOD approximation for $2$ coarse grid layers.}
\label{mp-2-serie}
\end{figure}

\begin{figure}[t]
\centering
\includegraphics[width=0.8\textwidth]{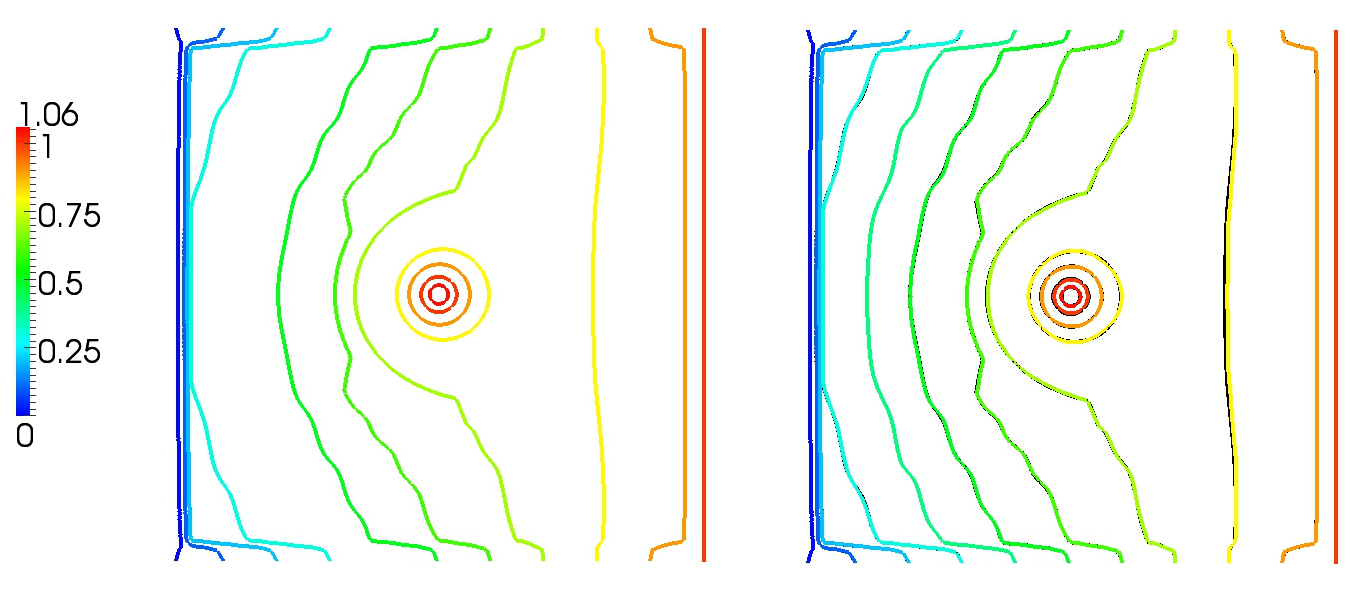}
\caption{\it Model problem 2. Computations made for $h=2^{-8}$ and $H=2^{-4}$ and $2$ coarse grid element layers for localization. The left picture depicts the isolines of the FEM reference solution on the fine grid (i.e. $h=2^{-8}$). The right picture shows a comparison of the isolines of LOD approximation and reference solution. The colored isolines belong to the LOD approximation. They overlie the corresponding black-colored isolines of the reference solution.}
\label{mp-2-isolines-serie}
\end{figure}

We consider the following model problem.
\begin{problem}
Let $\Omega:= ]0,1[^2$. Find $u \in H^1(\Omega)$ such that
\begin{equation*}\label{eq:model-2}
  \begin{aligned}
    -\nabla \cdot A(x)\nabla u(x) &= f(x)\quad \hspace{6pt}\text{\rm in }\Omega, \\
    u(x) & = x_1 \quad \hspace{17pt}\text{\rm on }\partial \Omega,
  \end{aligned}
\end{equation*}
where for $c:=(\frac{1}{2},\frac{1}{2})$ and $r:=0.05$
\begin{align*}
f(x):=\begin{cases} 20 \quad &\mbox{\rm if} \enspace |x-c|\le r\\
 0 \quad &\mbox{\rm else.}
\end{cases}
\end{align*}
The structure of the diffusion coefficient $A$ is depicted in Figure \ref{diffusion_problem_2}.
\end{problem}

The second model problem is devoted to the question on how the LOD is able to catch local properties of the exact solution (such as concentration accumulations) that are generated by an interaction of thin isolating channels and a contrasting boundary condition. The coarse grid is too coarse to capture the channels and too coarse to describe the narrow accumulations of the solution. Again, these effects must be captured and resembled by the local correctors. In model problem 2, the features of the exact solution are generated by a thin isolating frame close to the boundary of the domain (see Figure \ref{diffusion_problem_2}). Within the framed region the solution shows a different behavior to what is prescribed by the boundary condition. Furthermore, energy is pumped into the system by a very local source term $f$. The propagation is distorted by a circular structure that contains rings of high and low conductivity. Again, the FEM reference solution was obtained for a resolution of $h=2^{-8}$.

We start with a visual comparison that is depicted in Figure \ref{mp-2-serie}. The two plots on the left hand side of the figure show the reference solution. The middle and the right picture in the upper row show LOD approximations for $H=2^{-3}$ and the middle and the right picture in the lower row show LOD approximations for $H=2^{-4}$. In both cases, all desired features (in particular the steep and narrow accumulations) are captured by the LOD for patches with only $1$ coarse layer. The results are improved by adding another coarse layer. In this case, almost no difference to the reference solution is visible. This finding is emphasized by Figure \ref{mp-2-isolines-serie} where we can see a direct comparison of the isolines of reference and LOD approximation for $(h,H)=(2^{-8},2^{-4})$ and two coarse grid layers. The isolines are close to perfect matching. The error development in terms of coarse grid layers is given in Table \ref{serie-convergence-mp-2} for $H=2^{-3}$ and $H=2^{-4}$. Since Figures \ref{mp-2-serie} and \ref{mp-2-isolines-serie} predict that $2$ coarse grid layers are sufficient to obtain LOD approximations that are visually almost not distinguishable from the reference solution, this finding should also be recovered from the error table. Indeed, the results in Table \ref{serie-convergence-mp-2} show a fast error reduction within the first two coarse layers, then the error still decreases, but much slower. Adding a third or fourth coarse layer to the patches leads only to small improvements of the approximations and seems to be unnecessary. This finding is in accordance with the results of Theorem \ref{t:a-priori-local} which predict hat the first term in the a priori error estimate (which is of order $H \|f\|_{L^2(\Omega)}$ and $H^2 \|f\|_{L^2(\Omega)}$ respectively) will quickly dominate since the other terms decay exponentially to zero. For the results stated in Table \ref{serie-convergence-mp-2}, this dominance of the order $H$ term is already reached after $2$ coarse grid layers. The conductivity contrast of value $100$ does not lead to a demand for large patch sizes.

\subsection{Model problem 3}

\begin{figure}[h]
\centering
\includegraphics[width=0.8\textwidth]{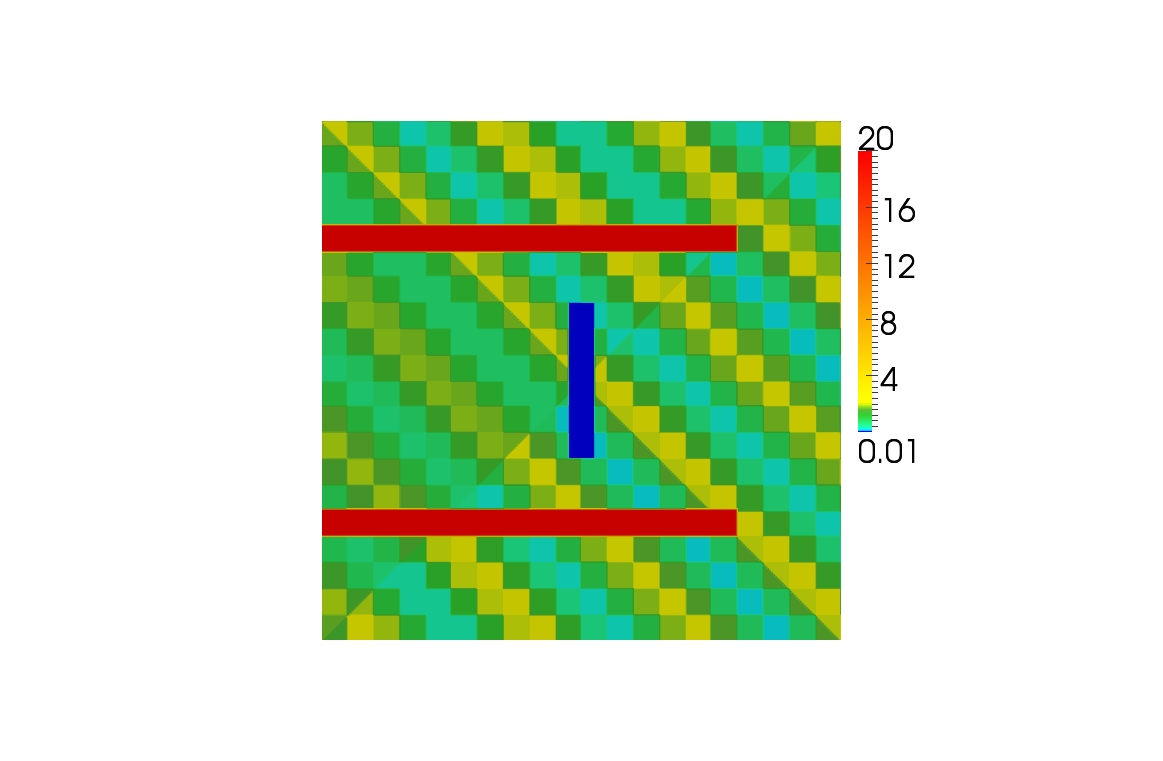}
\caption{\it Model problem 3. Plot of the diffusion coefficient $A$. It consists of a rapidly varying basis structure (green/turquoise/yellow region) given by the equation
$\frac{6}{5} + \frac{1}{2} \sin \left( \left\lfloor x_1 + x_2 \right\rfloor + \left\lfloor \frac{x_1}{\epsilon} \right\rfloor + \left\lfloor \frac{x_2}{\epsilon} \right\rfloor \right)
+ \frac{1}{2} \cos \left( \left\lfloor x_1 - x_2 \right\rfloor + \left\lfloor \frac{x_1}{\epsilon} \right\rfloor + \left\lfloor \frac{x_2}{\epsilon} \right\rfloor \right)$ for $\epsilon=0.05$.
Here, $A$ takes values between $0.2$ and $2.2$. This structure is perturbed by an isolator of thickness $\epsilon$ and length $0.3$ (blue region) and that has a conductivity of $0.01$. Additionally, there are two conductors (red) with high conductivity $20$. These two conductors also have a thickness of $\epsilon$ and a length of $0.8$. They are aligned with the Neumann inflow boundary condition given by (\ref{neumann-inflow-problem-3}).}
\label{diffusion_problem_3}
\end{figure}

\begin{table}[h]\footnotesize
\caption{\it Model problem 3. Computations made for $h=2^{-8}$, i.e. $|\mathcal{T}_h|=131072$ and $|\mathcal{N}_h|=66049$, and $1$ fixed coarse grid element layer for localization. In the second column, the number of fine grid element layers is shown that the coarse layer corresponds with. $|\mathcal{T}_h^\mathcal{U}|$ and $|\mathcal{N}_h^\mathcal{U}|$ are defined in (\ref{def-average-nodes-elements}). The table depicts $L^2$- and $H^1$-errors between $u_h$ and $u_{\LOD}$.}
\label{serie-1-coarse-layer-mp-3}
\begin{center}
\begin{tabular}{|c|c|c|c|c|c|c|}
\hline $H$     & \mbox{Fine layers} & k & $\|u_h - u_{\LOD}\|_{L^2(\Omega)}^{\mbox{\tiny rel}}$ & $\|u_h - u_{\LOD}\|_{H^1(\Omega)}^{\mbox{\tiny rel}}$ & $|\mathcal{T}_h^\mathcal{U}|$ & $|\mathcal{N}_h^\mathcal{U}|$ \\
\hline
\hline $2^{-2}$ & 64   &  1 & $0.02281$ & $0.23212$ & 49120 & 2488 \\
\hline $2^{-3}$ & 32   &  1 & $0.03547$ & $0.23215$ & 14696 & 7525 \\
\hline $2^{-4}$ & 16   &  1 & $0.02794$ & $0.28425$ & 3994   & 2090 \\
\hline $2^{-5}$ & 8     &  1 & $0.02104$ & $0.21349$ & 1037   & 566 \\
\hline
\end{tabular}
\end{center}
\end{table}

\begin{figure}[h]
\centering
\includegraphics[width=1.0\textwidth]{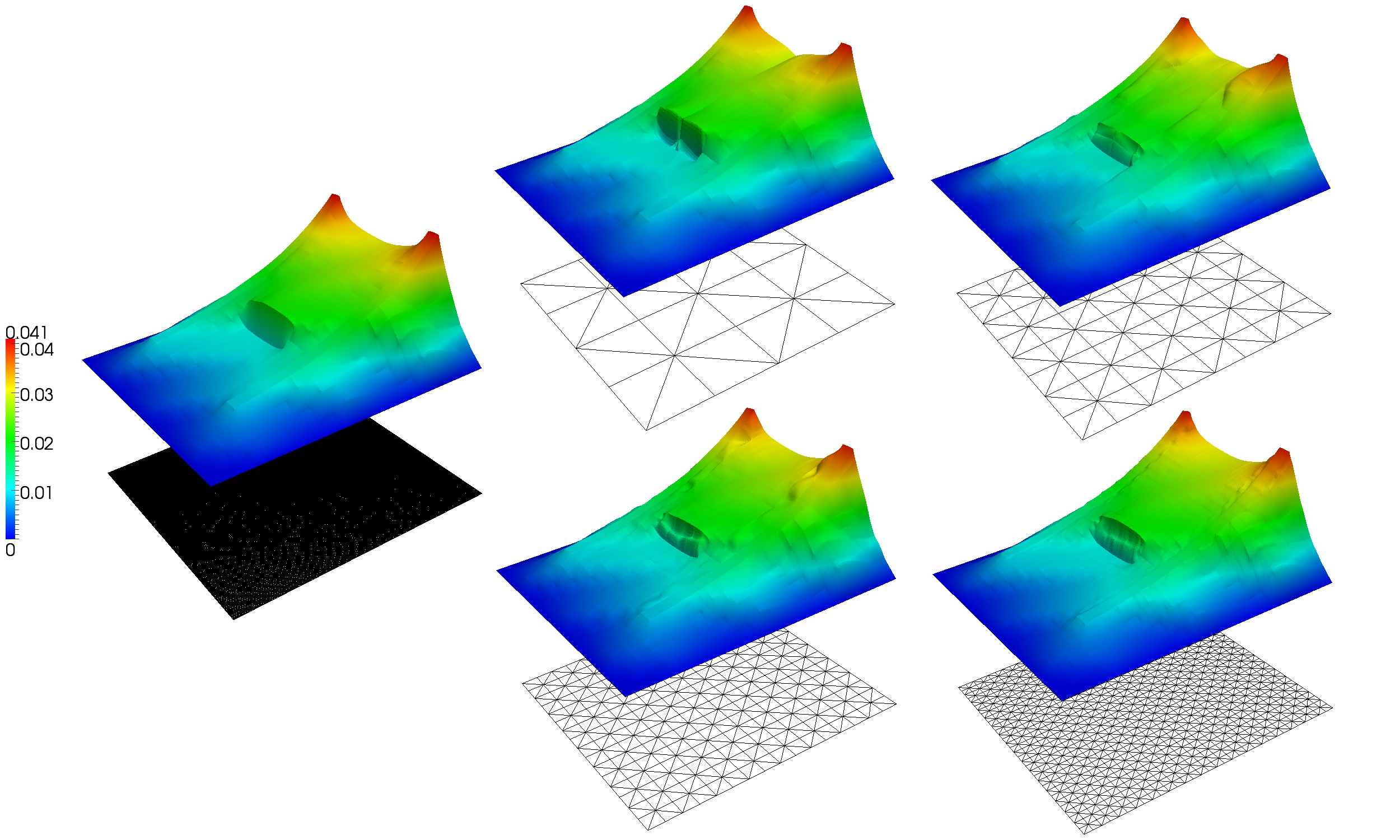}
\caption{\it Model problem 3. Computations made for $h=2^{-8}$ and $k=1$ fixed coarse grid element layer for localization. The left picture shows the FEM reference solution (for $h=2^{-8}$). The remaining pictures show the LOD approximations that correspond with the four results in Table \ref{serie-decay-mp-3}, i.e. they were obtained for the four different coarse grids that are mapped below the plots ($H=2^{-2},2^{-3},2^{-4},2^{-5}$).}
\label{mp-3-serie}
\end{figure}

\begin{table}[h]\footnotesize
\caption{\it Model problem 3. Computations made for  $H=2^{-3}$ and $h=2^{-8}$, i.e. $|\mathcal{T}_h|=131072$ and $|\mathcal{N}_h|=66049$. The first column depicts the number of fine grid element layers. $k$ denotes the corresponding number of coarse element layers. $|\mathcal{T}_h^\mathcal{U}|$ and $|\mathcal{N}_h^\mathcal{U}|$ are defined in (\ref{def-average-nodes-elements}). $L^2$- and $H^1$-errors between $u_h$ and $u_{\LOD}$ are shown.}
\label{serie-decay-mp-3}
\begin{center}
\begin{tabular}{|c|c|c|c|c|c|c|}
\hline \mbox{Fine layers} & k & $\|u_h - u_{\LOD}\|_{L^2(\Omega)}^{\mbox{\tiny rel}}$ & $\|u_h - u_{\LOD}\|_{H^1(\Omega)}^{\mbox{\tiny rel}}$ & $|\mathcal{T}_h^\mathcal{U}|$ & $|\mathcal{N}_h^\mathcal{U}|$ \\
\hline
\hline 4     &  0.125 & 0.21952 & 0.570727 & 2047   & 1102 \\
\hline 8     &  0.25   & 0.15593 & 0.528436 & 3290   & 1738 \\
\hline 16   &  0.5     & 0.09784 & 0.432237 & 6340   & 3291 \\
\hline 32   &  1        & 0.03547 & 0.232147 & 14696 & 7525 \\
\hline
\end{tabular}
\end{center}
\end{table}

Generally, multiscale methods such as HMM or MsFEM have the disadvantage that channels of high conductivity must be resolved by the macro grid in order to get reliable approximations. The reason is the following: if there are long channels of high conductivity in the computational domain, information is transported with high speed from one end of the channel to the other end. Now consider e.g. a local problem with a prescribed homogenous Dirichlet boundary condition on a patch. This problem is a localization of an originally global problem with homogenous Dirichlet boundary condition. Due to the high conductivity channel, the global solution can only decay to zero in a thin region very close to the boundary of the domain. Any interior localization of the solution that intersects the channel will not show a decay behavior. In other words, prescribing a zero boundary condition for a local function that cannot decay to zero on this patch leads to large discrepancy between chosen boundary condition on the patch and the real value on this boundary. The approximations are typically distorted and not reliable. However, if the coarse grid resolves these channel structures, the multiscale basis functions (for e.g. HMM or MsFEM) tend to standard finite element basis functions on the fine grid and the final approximation gets adequate again. An alternative is that the local problems are so large that they contain the full channels.

The situation for the LOD is different. Due to solving the corrector problems in a space that is the kernel of a quasi-interpolation operator, correctors show an intrinsic decay behavior that depends much less on the structure of the diffusion matrix $A$. Imagine that the Cl\'ement-type operator in the definition of $W_h$ (see (\ref{def-W_h})) is replaced by a nodal interpolation operator. Then $W_h$ consists of fine functions that are zero in every coarse grid node. This means that the solutions of the local problems lose repeatedly energy in these nodes. This leads to the previously stated exponential decay, even in the case of high conductivity channels. This consideration shall be emphasized by the following model problem, where we encounter two conductivity channels of width $\epsilon$ in which energy is brought in by a Neumann boundary condition. Additionally, we have a narrow isolator that forms a blockade. The model problem reads as follows.
\begin{problem}
Let $\Omega:= ]0,1[^2$, $\Gamma_N:=\{0\} \times ]0,1[$ and $\Gamma_D:=\partial \Omega \setminus \Gamma_N$.
\begin{equation*}\label{eq:model-3}
  \begin{aligned}
    -\nabla \cdot A\nabla u &= 0\quad \hspace{6pt}\text{\rm in }\Omega, \\
    u & = 0 \quad \hspace{7pt}\text{\rm on }\Gamma_D, \\
    A \nabla u \cdot n &= q \quad \hspace{7pt}\text{\rm on }\Gamma_N,
  \end{aligned}
\end{equation*}
where for $\epsilon:=0.05$
\begin{align}
\label{neumann-inflow-problem-3}q(0,x_2) := \begin{cases}
2 \quad &\mbox{\rm if} \enspace \quad 0.2 \hspace{18pt}\le x_2 \le 0.2 + \epsilon,\\ 
2 \quad &\mbox{\rm if} \enspace \quad 0.8-\epsilon \le x_2 \le 0.8,\\ 
 0 \quad &\mbox{\rm else.}
\end{cases}
\end{align}
The structure of the diffusion coefficient $A$ is depicted in Figure \ref{diffusion_problem_3}.
\end{problem}

We are interested in the behavior of the LOD in the case that none of the localization patches has 'full knowledge' about one of the conductivity channels, i.e. within each patch only a piece of information is accessible. For this purpose, we restrict ourselves to patches that contain maximum one coarse grid layer. We look at uniformly refined coarse grids with $H=2^{-2}$, $H=2^{-3}$, $H=2^{-4}$ and $H^{-5}$, i.e. we neither resolve the structure nor do we use large patches. The corresponding errors are presented in Table \ref{serie-1-coarse-layer-mp-3}, where each computation was performed for exactly one coarse layer. We observe that the $L^2$-errors (respectively $H^1$-errors) are all roughly of the same size. Note that convergence rates cannot be expected, since we fix the number of coarse layers (which leads to a strongly decreasing layer thickness). The results are equally good, independent of how much the coarse grid resolves the structures and independent of how much information from the channels is contained in the patches. This observation is stressed by Figure \ref{mp-3-serie}, where the reference solution and the corresponding LOD approximations are plotted. Each of the LOD approximations captures the information transported along the channels and the steep accumulation generated by the isolator. For $H=2^{-2}$ we see that the transitions are not yet fully smoothed but it improves with decreasing $H$. The approximation obtained for $H=2^{-5}$ comes visually very close to the reference solution. Finally, Table \ref{serie-decay-mp-3} shows that we still have the common error decay in terms of layers.

$\\$
{\bf{Acknowledgements.}} We would like to thank the anonymous reviewers for their constructive feedback and their very valuable and helpful suggestions.

\appendix
\section{Proof of Lemma \ref{lemma-influence-intersections}}

This proof is based in the arguments introduced in \cite{MaPe12} and \cite{Henning:Peterseim:2012} transferred to the general scenario of this work. We require the following lemma (see \cite[Lemma 2.1]{MaPe12} and \cite[Lemma 1]{Henning:Morgenstern:Peterseim:2013}) that characterizes the quasi-interpolation operator $I_H$:
\begin{lemma}
\label{lemma-2-1-from-MaPe12}
There exists a constant $C_1$ that can depend on the shape regularity of $\mathcal{T}_h$ and $\mathcal{T}_H$ but not on the mesh sizes $H$ and $h$, such that for all $v_H \in V_H$ there exists a $v_h \in V_h$ with
\begin{displaymath}
I_H(v_h)= v_H, \quad \|\nabla v_h \|_{L^2(\Omega)} \le C_1 \| \nabla v_H \|_{L^2(\Omega)}, \enspace\enspace\text{and} \enspace\enspace \mbox{\rm supp}(v_h) \subset \mbox{\rm supp}(v_H).
\end{displaymath}
\end{lemma}
\begin{proof}
The following proof can be found in a more detailed versions in \cite[Lemma 2.1]{MaPe12} and \cite[Lemma 1]{Henning:Morgenstern:Peterseim:2013}. For completeness, we add the main arguments.

For all coarse basis functions $\Phi_z \in V_H$, we search for $b_z \in V_h$ with
$$I(b_z)=\Phi_z,\quad|\nabla b_z(x)| \le C |\nabla \Phi_z(x)| \text{ for $x \in \Omega$ \enspace and}\quad\mbox{\rm supp}(b_z)\subset\mbox{\rm supp}(\Phi_z).$$
This can be achieved by choosing $b_z$ to be an element from the Finite Element space associated with the mesh given by a uniform refinement of $\mathcal{T}_H$, with values $0$ on $\partial(\mbox{supp}\Phi_z)$ and appropriately chosen values on the newly created nodes that can be determined by solving a system of equations (since we assumed that $\mathcal{T}_h$ was obtained by at least one uniform refinement of $\T_H$, $b_z$ will be an element of $V_h$). For an explicit construction of $b_z$, we refer the reader to \cite[Lemma 1]{Henning:Morgenstern:Peterseim:2013}. Finally, the function
$$v_h := v_H + \sum_{z\in\mathcal{N}_H}\left(v_H(z)-I_H(v_H)(z)\right) b_z \in V_h$$
has the desired properties.
\end{proof}

Recall the definition of the coarse layer patches $U_k(T)$ that were introduced in (\ref{def-patch-U-k}). We require suitable cut-off functions that are central for the proof. For $T\in\T_H$ and $\ell,k\in\mathbb{N}$ with $k > \ell$, we define $\eta_{T,k,\ell}\in V_H$ with nodal values
\begin{equation}\label{e:cutoffH}
\begin{aligned}
 \eta_{T,k,\ell}(z) &= 0\quad\text{for all }z\in\mathcal{N}\cap U_{k-\ell}(T),\\
 \eta_{T,k,\ell}(z) &= 1\quad\text{for all }z\in\mathcal{N}\cap \left(\Omega\setminus U_k(T)\right),\text{ and}\\
 \eta_{T,k,\ell} (z)&= \frac{m}{\ell}\quad\text{for all }x\in\mathcal{N}\cap \partial U_{k-\ell+m}(T),\;m=0,1,2,\ldots,\ell.
\end{aligned}
\end{equation}
For a given patch $\omega \subset \Omega$ also recall the definition $$\mathring{W}_h(\omega):=\{ v_h \in W_h| \hspace{2pt} v_h(z)=0 \enspace \mbox{for all } z \in \overline{\mathcal{N}_h \setminus \omega} \}.$$
  
We start with the following lemma, which says that $\eta_{T,k,\ell}w$ with $w \in W_h$ is close to a $W_h$-function.
\begin{lemma}\label{l:cutoff}
For a given $w\in W_h$ and a given cutoff function $\eta_{T,k,\ell}\in P_1(\T_H)$ defined in (\ref{e:cutoffH}) and $k>\ell>0$, there exists some $\tilde{w}\in \mathring{W}_h(\Omega \setminus U_{k-\ell-1}(T)) \subset W_h$  such that
\begin{equation}
\label{lemma-a-1-eq}\|\nabla(\eta_{T,k,\ell}w-\tilde{w})\|_{L^2(\Omega)}\lesssim  \ell^{-1} \|\nabla w\|_{L^2(U_{k+2}(T)\setminus U_{k-\ell-2}(T))}.
\end{equation}
\end{lemma}

\begin{proof}
We fix the element $T\in\T_H$ and $k\in\mathbb{N}$ and denote $\eta_\ell:=\eta_{T,k,\ell}$ and $c_K^{\ell} := |\omega_K|^{-1}\int_{\omega_K} \eta_\ell$ for $K\in \T_H$. Here, we define $\omega_K:=\cup\{K'\in\T_H\;\vert\;K'\cap K\neq\emptyset\}$. The operator $\mathcal{I}_h:H^1(\Omega)\cap C(\bar{\Omega})\rightarrow  P_1(\T_h)$ shall define the classical linear Lagrange interpolation operator with respect to $\T_h$. Lemma \ref{lemma-2-1-from-MaPe12} yields that there exists some $v\in V^h$ such that 
\begin{align}\label{e:lackproj}
\Ic v=\Ic \mathcal{I}_h(\eta_\ell w),\;\|\nabla v\|_{L^2(\Omega)}&\lesssim \|\nabla \Ic \mathcal{I}_h(\eta_\ell w)\|_{L^2(\Omega)},\;\text{and }\\\nonumber\operatorname{supp} (v)\subset \operatorname{supp} (\eta_\ell w) &\subset \Omega \setminus U_{k-\ell-1}(T).  
\end{align}
We can therefore define $\tilde{w}:=\mathcal{I}_h(\eta_\ell w)-v\in\oVf(\Omega \setminus U_{k-\ell-1}(T))$. Using (\ref{e:interr}) and $I_H(\mathcal{I}_h(w))=I_H(w)=0$ we obtain for any $K\in \T_H$
\begin{align}
\label{loc-lagrange-stability}\|\nabla \Ic \mathcal{I}_h(\eta_\ell w)\|_{L^2(K)} = \|\nabla \Ic \mathcal{I}_h((\eta_\ell -c_K^{\ell})w)\|_{L^2(K)} \lesssim \|\nabla ((\eta_\ell -c_K^{\ell})w)\|_{L^2(\omega_K)}.
\end{align}
Note that we used that the Lagrange interpolation operator $\mathcal{I}_h$ is $H^1$-stable on shape-regular partitions when it is restricted to piecewise polynomials of a fixed (small) degree (the stability constant only blows up to infinity, when the polynomial degree blows up to infinity, here the degree is bounded by $3$). This gives us
\begin{eqnarray}
\nonumber\lefteqn{\|\nabla \Ic \mathcal{I}_h(\eta_\ell w)\|_{L^2(\Omega)}^2}\\
\nonumber&\overset{\eqref{e:cutoffH},\eqref{loc-lagrange-stability}}{\lesssim}& \underset{K\subset \overline{U_{k+1}(T)}\setminus U_{k-\ell-1}(T)}{\sum_{K\in\T_H:}} \left\|\nabla\left(\left(\eta_\ell- c_K^{\ell} \right)w\right)\right\|_{L^2(\omega_K)}^2 \\
\nonumber&\lesssim& \underset{K\subset \overline{U_{k+1}(T)}\setminus U_{k-\ell-1}(T)}{\sum_{K\in\T_H:}}  \left\| (\nabla \eta_\ell)( w - I_H w) \right\|_{L^2(\omega_K)}^2 + \left\| \left(\eta_\ell- c_K^{\ell} \right)\nabla w \right\|_{L^2(\omega_K)}^2 \\
\nonumber&\overset{\eqref{e:cutoffH}}{\lesssim}& \underset{K\subset \overline{U_{k}(T)}\setminus U_{k-\ell}(T)}{\sum_{K\in\T_H:}} \hspace{-10pt} \left\| (\nabla \eta_\ell)( w - I_H w) \right\|_{L^2(K)}^2 \hspace{10pt}+ \hspace{-20pt} \underset{K\subset \overline{U_{k+1}(T)}\setminus U_{k-\ell-1}(T)}{\sum_{K\in\T_H:}} \hspace{-30pt} \left\| \left(\eta_\ell- c_K^{\ell} \right)\nabla w \right\|_{L^2(\omega_K)}^2 \\
\nonumber&\overset{(\ref{e:interr})}{\lesssim}& \|H \nabla \eta_\ell\|^2_{L^\infty(\Omega)} \| \nabla w  \|_{L^2(U_{k+1}(T)\setminus U_{k-\ell-1}(T))}^2 + \hspace{-20pt} \underset{K\subset \overline{U_{k+1}(T)}\setminus U_{k-\ell-1}(T)}{\sum_{K\in\T_H:}}\hspace{-30pt} \left\| \left(\eta_\ell- c_K^{\ell} \right)\nabla w \right\|_{L^2(\omega_K)}^2\\
\nonumber&{\lesssim}& \|H \nabla \eta_\ell\|^2_{L^\infty(\Omega)} \| \nabla w  \|_{L^2(U_{k+2}(T)\setminus U_{k-\ell-2}(T))}^2,\\
\label{e:ferlg}
\end{eqnarray}
where we used the Lipschitz bound $\| \eta_\ell- c_K^{\ell} \|_{L^{\infty}(\omega_K)} \lesssim H \| \nabla \eta_\ell \|_{L^{\infty}(\omega_K)}$. Recall the local $H^1$-estimate for the for the Lagrange interpolation operator on shape-regular partitions (c.f. \cite{Dubach:Luce:Thomas:2009} for quadrilaterals and hexahedra):
\begin{equation}\label{e:intesth}
 \|\nabla (p-\mathcal{I}_h p)\|_{L^2(S)}\lesssim h_S\|\nabla^2 p\|_{L^2(S)}
\end{equation}
for all $p\in C^0(\overline{S})\cap H^2(S)$ and $S\in\T_h$. Using this, $\mathcal{I}_h(w)=w$ and $I_H(w) = 0$ we get:
\begin{eqnarray}
\label{e:ferlg:2}\nonumber\lefteqn{\|\nabla(\eta_\ell w-\mathcal{I}_h(\eta_\ell w))\|_{L^2(\Omega)}^2 = \sum_{K \in \T_H} \|\nabla((\eta_\ell -c_K^{\ell}) w-\mathcal{I}_h((\eta_\ell -c_K^{\ell}) w))\|_{L^2(K)}^2 }\\
\nonumber&\lesssim& h^2 \hspace{-2pt} \sum_{K \in \T_H} \|\nabla^2 \eta_\ell (w-I_H(w))\|_{L^2(K)}^2 \hspace{-1pt} + \hspace{-1pt} \|\nabla \eta_\ell \cdot \nabla w\|_{L^2(K)}^2 \hspace{-1pt} + \hspace{-3pt} \underset{S\subset K}{\sum_{S\in \T_h:}} \|(\eta_\ell -c_K^{\ell}) \nabla^2 w\|_{L^2(K)}^2\\
\nonumber&\overset{(\ast)}{\lesssim}& h^2 \hspace{-10pt}\underset{K\subset \overline{U_{k+1}(T)}\setminus U_{k-\ell-1}(T)}{\sum_{K\in\T_H:}} \hspace{-15pt} \|\nabla \eta_\ell \|_{L^{\infty}(K)}^2 \| \nabla w \|_{L^2(K)}^2 + H^2 \| \nabla \eta_\ell \|_{L^{\infty}(\omega_K)}^2 \underset{S\subset K}{\sum_{S\in \T_h:}} h^{-2} \|\nabla w\|_{L^2(S)}^2\\
\nonumber&\lesssim& \|(h+H)\nabla\eta_\ell\|^2_{L^\infty(\Omega)} \| \nabla w  \|_{L^2(U_{k+1}(T)\setminus U_{k-\ell-1}(T))}^2\\
\end{eqnarray}
In ($\ast$) we used the obvious estimate $\|\nabla^2 \eta_{\ell} \|_{L^{\infty}(K)} \lesssim H^{-1} \|\nabla \eta_{\ell} \|_{L^{\infty}(K)}$ and the inverse estimate $\|\nabla^2 w\|_{L^2(S)} \lesssim h^{-1} \|\nabla w\|_{L^2(S)}$ (c.f. \cite{Brenner:Carstensen:2004}). Combing \eqref{e:ferlg} and \eqref{e:ferlg:2} yields:
\begin{eqnarray*}
\lefteqn{\|\nabla(\eta_\ell w-\tilde{w})\|_{L^2(\Omega)}^2\overset{\eqref{e:lackproj}}{\lesssim} \|\nabla(\eta_\ell w-\mathcal{I}_h(\eta_\ell w))\|_{L^2(\Omega)}^2 + \|\nabla \Ic \mathcal{I}_h(\eta_\ell w)\|_{L^2(\Omega)}^2}\\
&\overset{\eqref{e:ferlg},\eqref{e:ferlg:2}}{\lesssim}& \left(\|h\nabla\eta_\ell\|^2_{L^\infty(\Omega)}+\|H \nabla \eta_\ell\|^2_{L^\infty(\Omega)}\right) \|\nabla w\|^2_{L^2(U_{k+2}(T)\setminus U_{k-\ell-2}(T))}\\
&\overset{\eqref{e:cutoffH}}{\lesssim}& \ell^{-2}\|\nabla w\|^2_{L^2(U_{k+2}(T)\setminus U_{k-\ell-2}(T))}.
\end{eqnarray*}
This ends the proof.
\end{proof}

The following lemma describes the decay of the solutions of ideal corrector problems (i.e. problems such as (\ref{local-boundary-corrector-problem}) and (\ref{local-corrector-problem}) for $U(T)=\Omega$).
\begin{lemma}\label{l:decay}
Let $T\in \T_H$ be fixed and let $p_h^T \in W_h$ be the solution of
\begin{align}
\label{generalized-corrector-problem}\int_{\Omega} A \nabla p_h^T \cdot \nabla \phi_h =F_T(\phi_h) \qquad \mbox{for all } \phi_h \in W_h
\end{align}
where $F_T\in W_h^{\prime}$ is such that $F_T(\phi_h)=0$ for all $\phi_h \in \mathring{W}_h(\Omega \setminus T)$. Then, there exists a generic constant $0<\theta < 1$ (depending on the contrast) such that for all positive $k\in\mathbb{N}$:
\begin{align}
\label{lemma-a-2-eq}\|\nabla p_h^T \|_{L^2(\Omega\setminus U_k(T))}\lesssim \theta^{k}\|\nabla p_h^T \|_{L^2(\Omega)}.
\end{align}
\end{lemma}

\begin{proof}
The proof is analogous to \cite{MaPe12} and \cite{Henning:Peterseim:2012}. Let us fix $k\in\mathbb{N}$ and $\ell\in\mathbb{N}$ with $\ell< k-1$. We denote $\eta_\ell:=\eta_{T,k-2,\ell-4}\in P_1(\T_H)$ (as in \eqref{e:cutoffH}). 
Applying Lemma~\ref{l:cutoff} gives us the existence of $\tilde{p}_h^T\in  \mathring{W}_h(\Omega \setminus U_{k-\ell+1}(T))$ with $\|\nabla(\eta_\ell p_h^T -\tilde{p}_h^T)\|_{L^2(\Omega)}\lesssim  \ell^{-1} \|\nabla p_h^T\|_{L^2(U_{k}(T)\setminus U_{k-\ell}(T))}$. Due to $\tilde{p}_h^T\in  \mathring{W}_h(\Omega \setminus T)$ and the assumptions on $F_T$ we also have
\begin{align}
\label{lemm-4-4-step-1} \int_{\Omega \setminus U_{k-\ell}(T)} A \nabla p_h^T \cdot \nabla \tilde{p}_h^T =  \int_{\Omega} A \nabla p_h^T \cdot \nabla \tilde{p}_h^T = F_T ( \tilde{p}_h^T) = 0.
\end{align}
This leads to
\begin{eqnarray*}
\lefteqn{\int_{\Omega\setminus U_k(T)}A\nabla p_h^T \cdot \nabla p_h^T}\\
&\leq& \int_{\Omega\setminus U_{k-\ell}(T)}\eta_\ell A\nabla p_h^T \cdot \nabla p_h^T\\
&=& \int_{\Omega\setminus U_{k-\ell}(T)} A\nabla p_h^T\cdot \left(\nabla (\eta_\ell p_h^T)-p_h^T\nabla\eta_\ell\right)\\
&\overset{(\ref{lemm-4-4-step-1})}{=}& \int_{\Omega\setminus U_{k-\ell}(T)} A\nabla p_h^T\cdot \left(\nabla (\eta_\ell p_h^T-\tilde{p}_h^T)-(p_h^T-\underset{=0}{\underbrace{\Ic(p_h^T)}})\nabla\eta_\ell\right)\\
&\lesssim& \ell^{-1} \left(\|\nabla p_h^T \|_{L^2(\Omega \setminus U_{k-l}(T))}^2 + \right. \\
&\enspace& \qquad \left. H^{-1} \|\nabla p_h^T \|_{L^2(\Omega \setminus U_{k-\ell}(T))} \|p_h^T - I_H(p_h^T)\|_{L^2(\Omega \setminus U_{k-\ell}(T))} \right)\\
&\lesssim& \ell^{-1} \|\nabla p_h^T \|_{L^2(\Omega \setminus U_{k-\ell-1}(T))}^2.
\end{eqnarray*}
This implies that there exists a constant $C$ independent of $T$, $\ell$, $k$ and $A$, such that
\begin{equation}\label{e:decay}
\|\nabla p_h^T \|_{L^2(\Omega\setminus U_{k}(T))}^2\leq C \frac{\beta}{\alpha} \ell^{-1} \|\nabla p_h^T \|_{L^2(\Omega\setminus U_{k-\ell-1}(T))}^2.
\end{equation}
A recursive application of this inequality with the choice of $\ell:=\lceil eC\frac{\beta}{\alpha}\rceil$ yields
\begin{align*}
\|\nabla p_h^T \|_{L^2(\Omega\setminus U_{k}(T))}^2 \lesssim e^{-k/(\ell+3)} \|\nabla p_h^T \|_{L^2(\Omega)}^2.
\end{align*}
The choice $\theta:=e^{-(\lceil eC\frac{\beta}{\alpha}\rceil+3)^{-1}}$ proves the lemma.
\end{proof}

We are now prepared to prove the decay lemma:

\begin{proof}[Proof of Lemma \ref{lemma-influence-intersections}]
Again, the proof is analogous to \cite{Henning:Peterseim:2012}. We let $\eta_{T,k,1}$ be defined according to \eqref{e:cutoffH} and denote $z:=\sum_{T \in \T_H}(p_h^T-p_h^{T,k}) \in W_h$ (which again implies $I_H(z)=0$). We obtain
\begin{eqnarray*}
\lefteqn{\left\| A^{1/2} \nabla z \right\|_{L^2(\Omega)}^2}\\
&=& \sum_{T\in\T_H} \underset{=:\mbox{I}}{\underbrace{(A\nabla (p_h^T-p_h^{T,k}), \nabla (z (1- \eta_{T,k,1})))_{L^2(\Omega)}}} + \underset{=:\mbox{II}}{\underbrace{(A\nabla (p_h^T-p_h^{T,k}), \nabla(z \eta_{T,k,1}))_{L^2(\Omega)}}},
\end{eqnarray*}
where
\begin{eqnarray*}
\lefteqn{\mbox{I} \le \| \nabla (p_h^T-p_h^{T,k} )\|_{L^2(\Omega)} \| \nabla \left( z (1- \eta_{T,k,1}) \right) \|_{L^2(U_{k+1}(T))}}\\
&\le& \| \nabla( p_h^T-p_h^{T,k} )\|_{L^2(\Omega)} \left(  \| \nabla z \|_{L^2(U_{k+1}(T))} + \| z \nabla \left( 1- \eta_{T,k,1} \right) \|_{L^2(U_{k+1}(T) \setminus U_{k}(T) )} \right)\\
&\lesssim& \| \nabla( p_h^T-p_h^{T,k} )\|_{L^2(\Omega)} \left( \| \nabla z \|_{L^2(U_{k+1}(T))} + \frac{1}{H} \| z - I_H(z) \|_{L^2(U_{k+1}(T) \setminus U_{k}(T) )} \right)\\
&\lesssim& \| \nabla (p_h^T-p_h^{T,k}) \|_{L^2(\Omega)} \| \nabla z \|_{L^2(U_{k+2}(T))} 
\end{eqnarray*}
and again with Lemma \ref{l:cutoff} which gives us $\tilde{z}\in \mathring{W}_h(\Omega \setminus U_{k-2}(T))$ with the properties $\int_{\Omega} A \nabla (p_h^T-p_h^{T,k}) \cdot \nabla \tilde{z}=0$ and $\| \nabla (z \eta_{T,k,1} - \tilde{z}) \|_{L^2(\Omega)} \lesssim \| \nabla z \|_{L^2( U_{k+2}(T) )}$ and therefore
\begin{eqnarray*}
\mbox{II} = (A\nabla (p_h^T-p_h^{T,k}), \nabla((z \eta_{T,k,1})-\tilde{z})_{L^2(\Omega)}\lesssim \| \nabla (p_h^T-p_h^{T,k}) \|_{L^2(\Omega)} \| \nabla z \|_{L^2(U_{k+2}(T))}.
\end{eqnarray*}
Combining the estimates for I and II finally yields
\begin{eqnarray}
\label{lemma-a-3-proof-eq-0}\left\| A^{1/2} \nabla z \right\|_{L^2(\Omega)}^2
&\lesssim& \sum_{T\in\T_H} \| A^{1/2} \nabla (p_h^T-p_h^{T,k}) \|_{L^2(\Omega)} \| \nabla z \|_{L^2(U_{k+2}(T))}\\
\nonumber&\lesssim& k^{\frac{d}{2}} \left( \sum_{T\in\T_H} \| \nabla (p_h^T-p_h^{T,k}) \|_{L^2(\Omega)}^2 \right)^{\frac{1}{2}} \| A^{1/2} \nabla z \|_{L^2(\Omega)}.
\end{eqnarray}
It remains to bound $\| \nabla (p_h^T-p_h^{T,k}) \|_{L^2(\Omega)}^2$. In order to do this, we use Galerkin orthogonality for the  local problems, which gives us
\begin{align}
\label{galerkin-orthogonality-local-eq}\| \nabla (p_h^T-p_h^{T,k}) \|_{L^2(\Omega)}^2 \lesssim \inf_{\tilde{p}_h^{T,k} \in \mathring{W}_h(U_k(T))} \| \nabla (p_h^T-\tilde{p}_h^{T,k}) \|_{L^2(\Omega)}^2.
\end{align}
Again, we use Lemma \ref{lemma-2-1-from-MaPe12} which yields the existence of $\tilde{v}\in V^h$ such that 
\begin{align*}
\Ic \tilde{v}=\Ic \mathcal{I}_h((1-\eta_{T,k,1})p_h^T),\;\|\nabla \tilde{v}\|_{L^2(\Omega)}&\lesssim \|\nabla \Ic \mathcal{I}_h((1-\eta_{T,k,1})p_h^T)\|_{L^2(\Omega)},\;\text{and }\\\nonumber\operatorname{supp} (\tilde{v})\subset \operatorname{supp} ((1-\eta_{T,k,1})p_h^T) &\subset U_{k}(T).
\end{align*}
We can therefore define $\tilde{p}_h^T:=\mathcal{I}_h((1-\eta_{T,k,1})p_h^T)-\tilde{v}\in\oVf(U_{k}(T))$ and make two observations:
\begin{eqnarray}
\nonumber\lefteqn{\|\nabla \Ic \mathcal{I}_h((1-\eta_{T,k,1})p_h^T)\|_{L^2(U_k(T))}^2}\\
\nonumber&=& \|\nabla \Ic \mathcal{I}_h((1-\eta_{T,k,1})p_h^T)\|_{L^2(U_k(T)\setminus U_{k-2}(T))}^2 + \|\nabla \Ic \mathcal{I}_h((1-\eta_{T,k,1})p_h^T)\|_{L^2(U_{k-2}(T))}^2\\
\nonumber&=& \|\nabla \Ic \mathcal{I}_h((1-\eta_{T,k,1})p_h^T)\|_{L^2(U_k(T)\setminus U_{k-2}(T))}^2 + \|\nabla \Ic p_h^T\|_{L^2(U_{k-2}(T))}^2 \\
\nonumber&=& \|\nabla \Ic \mathcal{I}_h((1-\eta_{T,k,1})p_h^T)\|_{L^2(U_k(T)\setminus U_{k-2}(T))}^2\\
\label{lemma-a-3-proof-eq-1}
\end{eqnarray}
and
\begin{align}
\label{lemma-a-3-proof-eq-2}\| \nabla \left( (1-\eta_{T,k,1})p_h^T - \mathcal{I}_h((1-\eta_{T,k,1})p_h^T) \right) \|_{L^2(U_k(T)}^2 &\lesssim& \|\nabla p_h^T\|_{L^2(U_{k+1}(T) \setminus U_{k-2}(T))}^2,
\end{align}
which can be proved in the same way as equation (\ref{e:ferlg}) in Lemma \ref{l:cutoff}. Recall $\tilde{p}_h^T=\mathcal{I}_h((1-\eta_{T,k,1})p_h^T)-\tilde{v}$, then altogether we obtain
\begin{eqnarray}
\nonumber \lefteqn{\| \nabla (p_h^T-p_h^{T,k}) \|_{L^2(\Omega)}^2} \\
\nonumber  &\overset{(\ref{galerkin-orthogonality-local-eq})}{\lesssim}& \| \nabla (\eta_{T,k,1} p_h^T+ (1-\eta_{T,k,1})p_h^T - \tilde{p}_h^T)  \|_{L^2(\Omega)}^2 \\
\nonumber &\overset{(\ref{lemma-a-3-proof-eq-2})}{\lesssim}& \|\nabla p_h^T\|_{L^2(\Omega \setminus U_k(T))}^2 +  \|\nabla p_h^T\|_{L^2(U_{k+1}(T) \setminus U_{k-2}(T))}^2 +  \|\nabla \tilde{v}\|_{L^2(U_k(T))}^2\\
\nonumber &\lesssim& \|\nabla p_h^T\|_{L^2(\Omega \setminus U_k(T))}^2 +  \|\nabla p_h^T\|_{L^2(U_{k+1}(T) \setminus U_{k-2}(T))}^2 \\
\nonumber &\enspace& \quad +\|\nabla \Ic \mathcal{I}_h((1-\eta_{T,k,1})p_h^T)\|_{L^2(U_k(T))}^2\\
\nonumber &\overset{(\ref{lemma-a-3-proof-eq-1})}{\lesssim}& \|\nabla p_h^T\|_{L^2(\Omega \setminus U_{k-2}(T))}^2 + \|\nabla \Ic \mathcal{I}_h((1-\eta_{T,k,1})p_h^T)\|_{L^2(U_k(T) \setminus U_{k-2}(T))}^2\\ 
\nonumber &\lesssim& \|\nabla p_h^T\|_{L^2(\Omega \setminus U_{k-3}(T))}^2 \\
\nonumber &\overset{(\ref{lemma-a-2-eq})}{\lesssim}& \theta^{2 (k-3)} \|\nabla p_h^T \|_{L^2(\Omega)}^2.\\
\label{lemma-a-3-proof-eq-3}
\end{eqnarray}
Combining (\ref{lemma-a-3-proof-eq-0}) and (\ref{lemma-a-3-proof-eq-3}) proves the lemma.
\end{proof}

\end{document}